\documentclass[fleqn]{amsart}
\usepackage[utf8]{inputenc}
\usepackage{amssymb,latexsym}
\usepackage{amsmath}
\usepackage{calrsfs}
\usepackage{enumerate}

 \usepackage{amsfonts}
\usepackage{amsthm}
\usepackage{calrsfs}
\usepackage{enumerate}

\newcommand{\PG}{\mathrm{PG}}

\newcommand{\FF}{\mathbb{F}}

\newcommand{\KK}{\mathbb{K}}

\newcommand{\eps}{\varepsilon}
\newcommand{\teps}{\widehat{\varepsilon}}
\newcommand{\neps}{\varepsilon_{\mathrm{nat}}}
\newcommand{\tneps}{\widehat{\varepsilon}_{\mathrm{nat}}}
\newcommand{\ch}{\mathrm{char}}
\newcommand{\Der}{\mathrm{Der}}
\newcommand{\Kder}{N_{\mathrm{der}}}

\newcommand{\drk}{\mathrm{drk}}

\newcommand{\Tr}{\mathrm{Tr}}

\newcommand{\blank}{\,{\boldsymbol{\cdot}}\,}

\newtheorem{theorem}{Theorem}[section]
 \newtheorem{corollary}[theorem]{Corollary}

\newtheorem{claim}[theorem]{Claim}

\newtheorem{example}{Example}
\newtheorem{lemma}[theorem]{Lemma}

\newtheorem{proposition}[theorem]{Proposition}
\newtheorem{definition}[theorem]{Definition}
\newtheorem{remark}[example]{Remark}

\begin{document}
\title[The relatively universal cover of the adjoint embedding of $A_{n,\{1,n\}}$]{The relatively universal cover of the natural embedding of the long root geometry for the group $\mathrm{SL}(n+1,\KK)$}
\author{I. Cardinali}
\address[I. Cardinali, A. Pasini]{Dep. Information Engineering and Mathematics, University of Siena, Via Roma 56, I-53100 Siena, Italy}
\email{ilaria.cardinali@unisi.it}
\email{antonio.pasini@unisi.it}

\author{L. Giuzzi}
\address[L. Giuzzi]{DICATAM, University of Brescia, Via Branze 43, I-25123 Brescia, Italy}
\email{luca.giuzzi@unibs.it}

\author{A. Pasini}
%\address{Dep. Information Engineering and Mathematics, University of Siena, Via Roma 56, I-53100 Siena, Italy}


\maketitle

\begin{abstract}
	The long root geometry $A_{n,\{1,n\}}(\KK)$ for the special linear group $\mathrm{SL}(n+1,\KK)$ admits an embedding in the (projective space of) the vector space of the traceless square matrices of order $n+1$ with entries in the field $\KK$, usually regarded as the {\em natural} embedding of $A_{n,\{1,n\}}(\KK)$. S. Smith and H. V\"{o}lklein in~\cite{SV}
 have proved that the natural embedding of $A_{2,\{1,2\}}(\KK)$ is relatively universal if and only if $\KK$ is either algebraic over its minimal subfield or perfect with positive characteristic. They also give some information on the relatively universal embedding of $A_{2,\{1,2\}}(\KK)$ which covers the natural one, but that information is not sufficient to exhaustively describe it. The ``if" part of Smith-V\"{o}lklein's result also holds true for any $n$, as proved by  V\"{o}lklein in~\cite{Volklein} in his investigation of the adjoint modules of Chevalley groups.
In this paper we give an explicit description of the relatively universal embedding of $A_{n,\{1,n\}}(\KK)$ which covers the natural one. In particular, we prove that this relatively universal embedding has (vector) dimension equal to $\mathfrak{d}+n^2+2n$ where $\mathfrak{d}$ is the transcendence degree of $\KK$ over its minimal subfield (if $\ch(\KK) = 0$) or the generating rank of $\KK$ over $\KK^p$ (if $\ch(\KK) = p > 0$). Accordingly, both the ``if" and the ``only if" part of Smith-V\"{o}lklein's result hold true for every $n \geq 2$.
 \end{abstract}

\section{Introduction}\label{Introduction}

\subsection{Definitions, basics and a few known results}

Given a field $\KK$ and an integer $n \geq 2$, we denote by $A_{n,\{1,n\}}(\KK)$ the $\{1,n\}$-Grassmannian of a geometry of type $A_n$ defined over $\KK$. Explicitly, $A_{n,\{1,n\}}(\KK)$ is the point-line geometry defined as follows. Its points are the point-hyperplane flags $(p,H)$ of $\PG(n,\KK)$ and it admits two types of lines, namely the sets $\ell_{p,S}$ of all points $(p,X)$ of $A_{n,\{1,n\}}(\KK)$ with $S$ a given subspace of codimension $2$ of $\PG(n,\KK)$, $p$ a given point of $S$ and $X$ a generic hyperplane of $\PG(n,\KK)$ containing $S$ and the sets $\ell_{L,H}$ of all points $(x,H)$ of $A_{n,\{1,n\}}(\KK)$ with $H$ a given hyperplane of $\PG(V)$, $L$ a given line of $H$ and $x$ a generic point of $L$. In particular, when $n = 2$ the $2$-codimensional subspaces and the hyperplanes of the projective plane $\PG(2,\KK)$ are the points and the lines of $\PG(2,\KK)$. In this case $\ell_{p,S}$ is the set of flags $(p,X)$ of $\PG(2,\KK)$ with $X$ a line of $\PG(2,\KK)$ through  the point $p$ and $\ell_{L,H}$ is the set of flags $(x,L)$ with $x$ a point of the line $L$ of $\PG(2,\KK)$. The geometry $A_{n,\{1,n\}}(\KK)$ is called the {\it point-hyperplane geometry of} $\PG(n,\KK)$.

\subsubsection{Terminology for projective embeddings of point-line geometries}

We assume that the reader is familiar with the notions of projective embeddings of a point-line geometry (henceforth also called {\em embeddings} for short), and projections and isomorphisms of embeddings. Given two embeddings $\eps$ and $\eps'$ of the same geometry, we say that $\eps$ {\em covers} $\eps'$ if $\eps'$ is isomorphic to a projection of $\eps$. We recall that an embedding is
{\em relatively universal} if it is not covered by any other embedding;
it is {\em absolutely universal} if it covers any other embedding of the given geometry.

We refer the reader to Appendix~\ref{appendix A} for a synthetic survey of the relevant notions we will use and also to Ronan \cite{Ronan}, Shult \cite{SH93} and \cite[chp. 4]{SH11} and Kasikova and Shult \cite{K-S} for more on this topic. We only recall the following fact here: every projective embedding $\eps$ of a given geometry $\Gamma$ is covered by a relatively universal embedding $\teps$, uniquely determined up to isomorphisms and characterized by the following property: $\teps$ covers all embeddings of $\Gamma$ that cover $\eps$. We call $\teps$ the {\em (relatively) universal cover} of $\eps$ (see~\cite{SH93}).

\subsubsection{The natural embedding of $A_{n,\{1,n\}}(\KK)$}

Let $A$ be the $\KK$-vector space of the $(n+1)\times (n+1)$ matrices with entries in $\KK$ having null trace.
The \emph{natural embedding} of $A_{n,\{1,n\}}(\KK)$ is
the  projective embedding $\neps\colon A_{n,\{1,n\}}(\KK) \rightarrow \PG(A)$ mapping the point $(p,H)$ of $A_{n,\{1,n\}}(\KK)$, where $p$ is represented by a non-null (row) vector $v\in V = V(n+1,\KK)$ and $H$ is represented by a non-null linear functional $\alpha\in V^*$ (regarded as a column vector), to the point of $\PG(A)$ represented by the $(n+1)\times (n+1)$-matrix $\alpha\otimes v$.

Note that, as proved in \cite[Corollary 1.15]{AP-univ-emb}, when the field $\KK$ admits non-trivial automorphisms, then $A_{n,\{1,n\}}(\KK)$ admits no absolutely universal embedding. Nevertheless, the (relatively) universal cover of $\neps$ always exists. We shall denote it by $\tneps$.

The natural embedding  $\neps$ is {\it homogeneous}, which means that the full
automorphism group of $A_{n,\{1,n\}}(\KK)$ lifts to $\PG(A)$ as a group of collineations. In particular, with $G = \mathrm{SL}(n+1,\KK)$, the group $G$ acts (in general unfaithfully) on $A$ turning it into a $G$-module. The action of $G$ on $A$ is the adjoint one, that is every $g\in G$ maps the matrix $a\in A$ onto $a\cdot g:=g^{-1}ag\in A$.

Note that $A$ is just the Lie algebra of $G$, which is the Weyl module for $G$ associated with the highest root of the root system of type $A_n$ (whence the name {\em long root geometry of $G$}, often used for $A_{n,\{1,n\}}(\KK)$ in the literature). Clearly, $G$ acts as $\mathrm{PSL}(n+1,\KK)$ on both the geometry $A_{n,\{1,n\}}(\KK)$ and the vector space $A$.

The full automorphism group of $A_{n,\{1,n\}}(\KK)$ is contributed by all collineations and all dualities of $\PG(n,\KK)$. Hence $G$ acts on $A_{n,\{1,n\}}(\KK)$ as a proper subgroup of that group. In particular, $G$ does not act transitively on the set of lines of $A_{n,\{1,n\}}(\KK)$.  Indeed no element of $G$ switches the two families of lines of $A_{n,\{1,n\}}(\KK)$. In order to switch them we need a duality of $\PG(n,\KK)$. Nevertheless $G$ is close to being flag-transitive on $A_{n,\{1,n\}}(\KK)$. Indeed it acts transitively on the set of points of $A_{n,\{1,n\}}(\KK)$ and permutes transitively the lines of each of the two families. Moreover, if $\ell$ is a line of $A_{n,\{1,n\}}(\KK)$ then the stabilizer of $\ell$ in $G$ acts $2$-transitively on the set of points of $\ell$ and, for every point $x = (p,H)$ of $A_{n,\{1,n\}}(\KK)$, the stabilizer of $x$ in $G$ acts transitively on the set of lines through $x$ in each of the two families, namely the lines $\ell_{L,H}$ for $L$ a line of $H$ through $p$ and the lines $\ell_{p,S}$ for $S$ a hyperplane of $H$ containing $p$.

\subsubsection{The relatively universal cover of $\neps$}
Since $\neps$ is homogeneous, its relatively universal cover $\tneps\colon A_{n,\{1,n\}}(\KK) \rightarrow \PG(\widehat{A})$ is $G$-homogeneous as well, hence the vector space $\widehat{A}$ hosting the relatively universal cover $\tneps$ can be regarded as a $G$-module. The group $\widehat{G}$ induced by $G$ on $\PG(\widehat{A})$ stabilizes the $\tneps$-image $\tneps(A_{n,\{1,n\}}(\KK))$ of $A_{n,\{1,n\}}(\KK)$ and the transitivity properties of $G$ on $A_{n,\{1,n\}}(\KK)$ yield corresponding transitivity properties of $\widehat{G}$ on  $\tneps(A_{n,\{1,n\}}(\KK))$.

Let $x_0$ be a point of $A_{n,\{1,n\}}(\KK)$ and let $\ell_0$ and $\ell'_0$ be two lines
 on $x_0$, one from each of the two families of lines of $A_{n,\{1,n\}}(\KK).$ Moreover, let $\widehat{p}_0$ be a point of $\PG(\widehat{A})$ and $\widehat{L}_0$ and $\widehat{L}'_0$ be two lines of $\PG(\widehat{A})$ through $\widehat{p}_0$ such that the projection of $\widehat{A}$ onto $A$ maps $\widehat{p}_0$, $\widehat{L}_0$ and $\widehat{L}'_0$ onto $\neps(x_0)$, $\neps(\ell_0)$, $\neps(\ell'_0)$ respectively. If we know how $\widehat{G}$ acts on $\widehat{A}$ then $\tneps(A_{n,\{1,n\}}(\KK))$ is basically known. Indeed, for every point $x$ of $A_{n,\{1,n\}}(\KK)$, we have $x = g(x_0)$ for some
$g\in G$. Accordingly, we can assume that $\tneps(x) = \widehat{g}(\tneps(x_0))$, where $\widehat{g}$ is the lifting of $g$ to $\widehat{A}$. Similarly for lines.

We call a pair of triples  $\{(x_0, \ell_0, \ell'_0), (\widehat{p}_0, \widehat{L}_0, \widehat{L}'_0)\}$ as above a {\em pivot} for $\tneps$.  Clearly, changing the pivot amounts to shifting $\tneps$ by a collineation of $\PG(\widehat{A})$.

Of course, in this way the embedding $\tneps$ is recovered up to collineations of $\PG(\widehat{A})$ which stabilize every fiber of the projection $\pi\colon \widehat{A}\rightarrow A.$ We can then replace $\tneps$ with $\gamma\circ\tneps$ for a collineation $\gamma$ of $\PG(\widehat{A})$ such that $\pi\circ \gamma=\pi.$

For $n=2$, Smith and V\"{o}lklein \cite{SV} prove that the module $\widehat{A}$ is a central extension of $A$, that is $\widehat{A}=M\cdot A$ where $G$ acts trivially on $M$. Then, relying on the theory of cohomology, they prove that the dual $M^*$ of $M$ is isomorphic to the first cohomology group $H^1(A^*,G)=Z^1(A^*,G)/B^1(A^*,G),$ obtaining that $\neps = \tneps$ if and only if $M^*=0$ if and only if $\KK$ is algebraic over its minimal subfield or perfect of positive characteristic. Relying on this result, V\"{o}lklein \cite[Corollary 1]{Volklein} proves that, if $\KK$ has these properties, then, for every split Chevalley group $G$ defined over $\KK$ but not of Dynkin type $C_n$, the adjoint module of $G$ affords a relatively universal projective  embedding of the long root geometry associated to $G$ (the $C_n$-case is excluded because in this case the `adjoint embedding' of the long-root geometry is veronesean instead of projective). In particular, if $\KK$ is either algebraic or perfect of positive characteristic then the natural embedding $\neps$ of $A_{n,\{1,n\}}(\KK)$ is relatively universal.

\subsection{The main result of this paper}\label{Intro2}
We shall give an explicit description of the relatively universal cover $\tneps:A_{n,\{1,n\}}(\KK)\rightarrow\PG(\widehat{A})$ of the natural embedding $\neps : A_{n,\{1,n\}}(\KK)\rightarrow \PG(A)$ for any $n$. Generalizing the result obtained by Smith and V\"{o}lklein \cite{SV} for the case $n = 2$, we firstly prove the following.

\begin{lemma}\label{main lemma}
As a $G$-module, $\widehat{A}$ is a non-split central extension $\widehat{A} = M\cdot A$ of $A$ and $M^* \cong H^1(A^*, G)$.
\end{lemma}
The vector space $H^1(A^*,G)$ is in turn isomorphic to the space $\mathrm{Der}(\KK)$ of all derivations of $\KK$ (V\"{o}lklein \cite{V}, also Theorem \ref{C1} of this paper). So, the dual $M^*$ of the space $M$ we are looking for must be isomorphic to $\mathrm{Der}(\KK)$.

The vector space $\mathrm{Der}(\KK)$ can be described as follows. Let $\Kder(\KK)$ be the largest subfield of $\KK$ such that all derivations of $\KK$ induce the trivial derivation on it. A subset $\Omega$ of $\KK\setminus\Kder(\KK)$ is a {\em derivation basis} of $\KK$ if for every mapping $\nu:\Omega\rightarrow\KK$ there exists a unique derivation $d_\nu\in\mathrm{Der}(\KK)$ such that $d_\nu$ induces $\nu$ on $\Omega$. Referring to Appendix~\ref{appendix B} for more details, here we only recall the following: every field admits derivation bases and all derivation bases of a given field have the same cardinality. So, given a derivation basis $\Omega$ of $\KK$, the vector space $\mathrm{Der}(\KK)$ is canonically isomorphic to the vector space $\KK^\Omega$ of all $\KK$-valued functions on $\Omega$. Thus, we are looking for a $\KK$-vector space $M$ such that $M^*\cong \KK^\Omega$.

With $\Omega$ as above, for every $\omega\in \Omega$ let $d_\omega$ be the derivation of $\KK$ such that $d_\omega(\omega) = 1$ and $d_\omega(\omega') = 0$ for every $\omega'\in\Omega\setminus\{\omega\}$. The set $\{d_\omega\}_{\omega\in\Omega}$ is linearly independent in $\Der(\KK)\cong \KK^\Omega$ and spans a subspace $\Der_\Omega(\KK)$ of $\Der(\KK)$ isomorphic to the subspace $\KK^{\Omega|\mathrm{fin}}$ of $\KK^\Omega$ formed by the functions with finite support. So $\Der(\KK)$ is isomorphic to the dual of $\Der_\Omega(\KK)$, since $\KK^{\Omega}$ is isomorphic to the dual of $\KK^{\Omega|\mathrm{fin}}$.

Assuming that $M = \Der_\Omega(\KK)$ is the vector space we are looking for, so that $\widehat{A} = M\times A$ as a vector space, we have to find out how the group $G$ acts on $M\times A$. Theorem \ref{thm link} of this paper and the isomorphism $\Der(\KK)\cong H^1(A^*,G)$ as described in Theorem \ref{C1} suggest to try the following action:

\begin{equation}\label{ovA}
g~:~ (m, a)\in M\times A ~\longrightarrow ~ (m + \sum_{\omega\in\Omega}\mathrm{Tr}(g\cdot d_\omega(g^{-1})\cdot a)d_\omega, ~ g^{-1}ag  ) \in M\times A
\end{equation}
where $a\in A$ and $g\in G$ are regarded as matrices, $d_\omega(g^{-1})$ is the matrix obtained by applying $d_\omega$ to each of the entries of the matrix which represents $g^{-1}$, $g\cdot d_\omega(g^{-1})\cdot a$ is the product of the matrices $g$, $d_\omega(g^{-1})$ and $a$ and $\mathrm{Tr}(g\cdot d_\omega(g^{-1})\cdot a)$ is the trace of the matrix obtained in that way.

Note that $d_\omega(g^{-1})$ is the null matrix for all but at most finitely many choices of $\omega\in\Omega$ since, as we shall show in Appendix \ref{appendix B} (Proposition B2), for every $k\in \KK$ we have $d_\omega(k)\neq 0$ for finitely many (possibly no) choices of $\omega\in\Omega$. Hence only a finite number of non-zero summands are involved in the sum $\sum_{\omega\in\Omega}\mathrm{Tr}(g\cdot d_\omega(g^{-1})\cdot a)d_\omega$. So, even when $\Omega$ is infinite, that sum makes sense.

The following, to be proved in Section~\ref{Sec 3}, is our main theorem.

 \begin{theorem}\label{main theorem}
As a $G$ module, $\widehat{A}$ is the same as the $G$-module defined by {\rm \eqref{ovA}} on $M\times A$ where $M = \Der_\Omega(\KK)$ and, up to collineations of $\PG(\widehat{A})$, we can always assume to have chosen the pivot of the embedding $\tneps:A_{n,\{1,n\}}(\KK)\rightarrow\PG(\widehat{A})$ in such a way that $\tneps$ acts as follows.

For every point $(p,H)$ of $A_{n,\{1,n\}}(\KK)$, if $v\in V = V(n+1,\KK)$ represents $p$ and $\lambda\in V^*$ represents $H$, then $\tneps$ maps $(p,H)$ onto the point of $\PG(\widehat{A})$ represented by the vector
\begin{equation}\label{hatA}
(\sum_{\omega\in\Omega}v\cdot d_\omega(\lambda)\cdot d_\omega, ~ \lambda\cdot v)
\end{equation}
where, regarded $\lambda = (\lambda_i)_{i=1}^{n+1}$ as a (column) vector, we put $d_\omega(\lambda) := (d_\omega(\lambda_i))_{i=1}^{n+1}$.
\end{theorem}
(Of course, the scalar $v\cdot d_\omega(\lambda)$ in \eqref{hatA} is the value $d_\omega(\lambda)(v)$ taken by $d_\omega(\lambda) \in V^*$ on $v$). As we shall show in the Appendix \ref{appendix B}, we have $\mathrm{Der}(\KK) = \{0\}$ if and only if $\KK$ is either algebraic over its minimal subfield or perfect of positive characteristic. Therefore,

 \begin{corollary}\label{main corollary 2}
The natural embedding $\neps$ of $A_{n,\{1,n\}}(\KK)$ is relatively universal if and only if $\KK$ is algebraic over its minimal subfield or  perfect with positive characteristic.
\end{corollary}

\begin{remark}\label{rem1}
\em
An explicit description of $\widehat{A}$ is missing in \cite{SV} even for $n = 2$, which is the case that paper is devoted to. The authors of \cite{SV} prove that, when $n = 2$, the $G$-module $\widehat{A}$ is an extension of $A$ by a trivial $G$-module $M$ and $M^*\cong H^1(A^*,G)$ ($\cong \Der(\KK)$), but they neither describe $M$ when $\dim(M^*)$ is infinite nor explain how $G$ acts on $\widehat{A}$. All they say on this action is a formula which  follows from formula (2.5) of \cite{SV} and should describe the action of $G$ on $\widehat{A}/H$ for a generic hyperplane $H$ of $M$. However that formula is incorrect (it even fails to define an action of $G$). Indeed both in that formula and in (2.5) the authors wrongly write $g^{-1}\cdot d(g)$ instead of $g\cdot d(g^{-1})$ (compare our formula \eqref{ovA}). The correction of that error has been the  starting point of our paper.
\end{remark}

\noindent
{\bf Organization of the paper.} In Section 2 we prepare the tools to be used in the proof of Lemma \ref{main lemma} and Theorem \ref{main theorem}. In the first part of the section we recall some generalities on extensions of modules for a given group $G$. In the second part, given a geometry $\Gamma$ admitting an action of $G$ as a subgroup of $\mathrm{Aut}(\Gamma)$ and a $G$-module $V$ hosting a $G$-homogeneous embedding $\eps$ of $\Gamma$, properties of the extensions of $V$ which host a cover of $\eps$ are discussed. Section 3 contains the proofs of Lemma \ref{main lemma} and Theorem \ref{main theorem}. We close our paper with three appendices. In the first appendix we recall a construction of the relatively universal cover of a given embedding, due to Ronan \cite{Ronan}. In the second one we give all information on $\Der(\KK)$ to be used in this paper and in the third one we offer an explicit description of the isomorphism $\Der(\KK)\cong H^1(A^*,G)$.

\section{Module extensions and projective embeddings}\label{Sec 2}

Throughout this section $G$ is a given group and $\KK$ a given field. A right (left) $G$-module is a pair $(V,\rho)$ where $V$ is a $\KK$-vector space and $\rho$ is a homomorphism from $G$ to the group of invertible linear transformations of $V$ such that $\rho(g_1g_2)(v) = \rho(g_2)(\rho(g_1)(v))$ (respectively $\rho(g_1)(\rho(g_2)(v))$) for any two elements $g_1, g_2\in G$ and every $v\in V$. By a little abuse, if a homomorphism $\rho$ as above is given we say that $V$ is a $G$-module, for short. If $V$ is a right (left) $G$-module we write $v\cdot g$ (respectively $g\cdot v$) for $\rho(g)(v)$, for every $g\in G$ and $v\in V$. We extend this notation to subspaces of $V$ and subgroups of $G$ in a natural way. Explicitly, let $V$ be a right $G$-module, $X$ and $v$
respectively a subset and a vector of $V$ and $F$ and $g$ a subset and an element of $G$. Then we put $X\cdot g : = \{v\cdot g\}_{v\in X}$, $vF := \{vg\}_{g\in F}$ and $X\cdot F = \cup_{v\in X}v\cdot F = \cup_{g\in F}X\cdot g$. In particular, $v\cdot G$ is the  \emph{orbit} of $v$ under $G$ and if $X\cdot G = X$ then $G$ is said to  \emph{stabilize} $X$. We will also take the liberty of writing $g(v)$, $g(X)$, $F(v)$ and $F(X)$ instead of  $v\cdot g$, $X\cdot g$, $v\cdot F$ and $X\cdot F$ respectively, but only when this free notation can make our formula easier to read.

All modules to be considered in the sequel are right $G$-modules. In particular, given a (right) $G$-module $V$, the action of $G$ on the dual $V^*$ of $V$ is defined as follows:
\begin{equation}\label{GsuV*}
\xi\cdot g ~:~ x\in V ~ \longrightarrow ~ \xi(x\cdot g^{-1}), \hspace{5 mm} \forall g\in G, ~\forall \xi \in V^*.
\end{equation}
In short, $\xi\cdot g = \xi\circ g^{-1}$, where $\circ$ stands for composition of mappings. According to this definition, we have $\xi\cdot(g_1g_2) = (\xi\cdot g_1)\cdot g_2$ for every choice of $\xi\in V^*$ and $g_1, g_2\in G$. So, given a right $G$-module structure on $V$, definition \eqref{GsuV*} indeed makes $V^*$ a right $G$-module.

\begin{remark}
\em
Let $V = V(n,\KK)$, the vectors of $V$ and those of $V^*$ being regarded as $1\times n$ matrices and $n\times 1$ matrices respectively. Let $G$ be a group of $n\times n$ matrices acting on $V$ (on the right) in the natural way: $x\cdot g$ is just the product of the matrix $x$ times the matrix $g$. According to \eqref{GsuV*}, the $g$-image $\xi\cdot g$ of a linear functional $\xi\in V^*$ is the product $g^{-1}\xi$ of the inverse of the matrix $g$ times the matrix $\xi$. As the matrix $g$ is replaced by its inverse, the group $G$ indeed acts on the right on $V^*,$ even if $g^{-1}$ occurs on the left in the matrix product $g^{-1}\xi$.
\end{remark}

\subsection{Generalities on extensions of $G$-modules}

Given a $G$-module $U$, a  \emph{submodule} of $U$ is a subspace of $U$ stabilized by $G$ in its action on $U$. If $W$ is a submodule of $U$ then an action of $G$ is naturally defined on the quotient $U/W$ as follows: $(u+ W)\cdot g := u\cdot g + W$ for every coset $u+W$ of $W$ and every $g\in G$.

Given another $G$-module $V$, the module $U$ is an \emph{extension} of $V$ if $U$ admits a submodule $W$, called the  \emph{kernel} of the extension, such that $V\cong U/W$ (isomorphism of $G$-modules). We write $U = W\cdot V$ to mean that $U$ is an extension of $V$ with kernel $W$.

If $\dim(W)=d$ we say that $U$ is a $d$-\emph{extension} of $V$. In particular, if $d = 1$ then $U$ is a $1$-\emph{extension} of $V$.
In case $W=\{0\}$ the extension $U$ is \emph{trivial}.

We say that a non-trivial extension $U = W\cdot V$  \emph{splits} if there is an injective linear mapping $\phi:V\to U$ such that $\phi(V)$ is stabilized by $G$ and $\pi\circ\phi= \mathrm{id}_V$, where $\pi$ is the canonical projection of $U$ onto $U/W \cong V$ as vector spaces. If this is the case then $U$ is the direct sum of $W$ and $\phi(V)$ (direct sum of $G$-modules, of course); with a little abuse, $U = W\oplus V$ (equivalently $U = W\times V$ as only two summands are involved). An extension $U$ is said to be  \emph{non-split} if it does not split.

Given a subspace $W_0$ of $W$ stabilized by $G$, the quotient $U/W_0$ is again an extension of $V$ called a  \emph{quotient} of $U$ (a  \emph{proper quotient} if $W_0 \subsetneq W$).

\begin{definition}\label{def0}
\em
We say that the extension $U = W\cdot V$ of $V$ is  \emph{totally non-split} if every proper quotient of $U$ is non-split as an extension of $V.$
\end{definition}

The following definitions will also be used in this paper.

\begin{definition}\label{def1}
\em
Let $U=W\cdot V$ be an extension of $V.$ We say that the extension $U$ is
\begin{itemize}
\item \emph{central} if $G$ acts trivially on the kernel $W$ of $U$,
\item $1$-\emph{complete} if $U$ admits no non-split $1$-extension,
\item \emph{$1$-universal} if every non-split $1$-extension of $V$ is isomorphic to a quotient of $U$ over a hyperplane of $W$ stabilized by $G$,
\item $1$-\emph{non-split} if $U/H$ is a non-split extension of $V$ for every hyperplane $H$ of $W$ stabilized by $G$.
  \end{itemize}
\end{definition}

\begin{remark}\label{Gperf-1-et}
\em
If a group acts on a $1$-dimensional $\KK$-vector space, then it
necessarily acts on it as a subgroup of the multiplicative group of $\KK$. Consequently, the trivial action is the unique action of a perfect group on a $1$-dimensional vector space. So, if $G$ is perfect then all $1$-extensions of
a $G$-module $V$ are central.
\end{remark}

Let $U=W\cdot V$ be an extension of $V$ and suppose $X$ is a complement of $W$ in the vector space $U$. The elements of $V$ are then cosets of $W$, so for every $v\in V$ the intersection $v_X:=v\cap X$ is an element of $X.$
The map $\iota_X \colon V \rightarrow X,\, \iota_X(v)= v_X$ is a vector space  isomorphism.

For every $g\in G$, we define an action $g_X$ of $g$ on $X$ as follows:
\begin{equation}\label{gX}
g_X ~\colon ~ x\in X ~\longrightarrow ~((x+W)\cdot g)_X.
\end{equation}
Since $(x+W)\cdot g  = x\cdot g + W$ (with $x\cdot g$ computed in $U$) and $w_X =w\cap X= 0$ for every $w\in W$, this action is well defined. According to \eqref{gX}, the vector space $X$ can be regarded as a $G$-module and $\iota_X$ is a module isomorphism.

If $X$ is not stabilized by $G$ in the action of $G$ on $U$, the subspace $X$ is not a submodule of $U.$ However $v_X\cdot g -  (v\cdot g)_X \in W$ for every $v\in V$ and $g\in G$. We put
\begin{equation}\label{wX}
w_X(v,g) ~ :=  ~ v_X\cdot g-(v\cdot g)_X
\end{equation}
and  define $W_X$ as the subspace of $W$ spanned by the elements $w_X(v,g)$'s for $v\in V$ and $g\in G$. The group $G$ acts as follows on $W_X:$
\begin{equation}\label{cociclo}
	w_X(v,g_1)\cdot g_2 ~ = ~ w_X(v,g_1g_2)-w_X(vg_1,g_2).
\end{equation}
Also, according to \eqref{wX}
\begin{equation}\label{pre cociclo}
w_X(v,1) ~ = ~  v_X - v_X = 0
\end{equation}
for every $v\in V$. (Note that the  equality $w_X(v,1) = 0$ can also be deduced formally from \eqref{cociclo}, simply putting $g_1 = 1$ in that formula).  As a consequence of equation~\eqref{cociclo}, we have the following.

\begin{lemma}\label{WX}
The group $G$ stabilizes the subspace $W_X$ of $W$.
\end{lemma}

The subspace $X$ is stabilized by $G$ if and only if $W_X=\{0\}.$ So, $U$ is non-split  if and only if $W_X\not=\{0\}$ for every complement $X$ of $W$. Combining these remarks with Lemma \ref{WX} we immediately obtain the following.

\begin{proposition}
The extension $U=W\cdot V$ is totally non-split if and only if $W_X=W$ for every complement $X$ of $W$ in $U.$
\end{proposition}

\begin{proposition}\label{1-ext}
 A central extension $U=W\cdot V$ of $V$ is totally non-split if and only if it is $1$-non-split.
\end{proposition}
\begin{proof}
The `only if' part of the statement is trivial. We prove the `if' part by contraposition. Suppose there exists a proper subspace $W_0$ of $W$ such that $U/W_0$ splits as $U/W_0=W/W_0\times X$ where $X\cong V.$ So, if $H$ is a hyperplane of $W$ containing $W_0$ (necessarily stabilized by $G$ because $G$ acts trivially on $W$ since $U$ is a central extension by hypothesis) then the extension  $U/H=W/H \cdot V$ splits as $U/H=W/H\times X$.
\end{proof}

\subsection{Central $1$-extensions}
As Smith and Volklein recall in~\cite[pag. 134]{SV} the central $1$-extensions of a $G$-module $V$ correspond in a standard way to the elements of the first cohomology group $H^1(V^*,G)$ of $G$ over the dual $V^* $ of $V$. In this subsection we shall discuss this claim. We shall also recall its proof, since in this case the proof is more enlightening than the statement itself. We are not going to recall basics on cohomology of groups. We refer to Hall \cite[chapter 15]{Hall} for them. Of course, many more expositions are available on this matter, but \cite{Hall} is enough for our needs.

\subsubsection{Preliminaries}

Let $U = W\cdot V$ be a central extension of $V$
and $X$ be a complement of $W$ in $U$.
As $G$ acts trivially on $W$ equation~\eqref{cociclo} yields
\begin{equation}\label{wXgv}
	w_X(v,g_1)=w_X(v,g_1g_2)-w_X(vg_1,g_2), ~~~ \forall g_1, g_2\in G, ~\forall v\in V.
\end{equation}
For every $v\in V$ define $w^*_X(v,\blank):G\rightarrow W$ as follows:
 \begin{equation}\label{cociclo bis}
w^*_X(v,g)~ := ~ w_X(v,g^{-1}), ~~~ \forall g\in G.
\end{equation}
With this notation, we can rewrite \eqref{cociclo} as follows:
\[w^*_X(v, g_2^{-1}g_1^{-1}) ~ = ~  w^*_X(vg_1, g_2^{-1}) + w^*_X(v, g_1^{-1}), ~~ \forall g_1, g_2\in G; ~ \forall v\in V.\]
Replacing $g_1^{-1}$ with $f_1$ and $g_2^{-1}$ with $f_2$ and next changing our notation again, renaming $f_1$ and $f_2$ as $g_2$ and $g_1$ respectively, the above yields
\begin{equation}\label{cociclo ter}
w^*_X(v, g_1g_2) ~ = ~ w^*_X(vg_2^{-1}, g_1) + w^*_X(v, g_2).
\end{equation}
Given $g\in G$, we can regard $w^*_X(\blank ,g)$ as a linear mapping $f_X(g):v\rightarrow w^*_X(v,g)$ from $V$ to $W$ (in fact to $W_X$). Accordingly, $w^*_X(\blank, \blank )$ can be regarded as the mapping $f_X$ from $G$ to the space $L(V, W)$ of linear mappings from $V$ to $W$ which maps every $g\in G$ onto $f_X(g)$. With this notation, the mapping $w^*_X((\blank)g_2^{-1}, g_1) : v \rightarrow w^*_X(vg_2^{-1}, g_1)$ is the same as $f_X(g_1)\cdot g_2$. (Recall that if $f$ is a mapping from a $G$-module $V$ to a set $S$ and $g\in G$ then $f\cdot g = f\circ g^{-1}$; compare \eqref{GsuV*}). With this notation we can rewrite \eqref{cociclo ter} as follows
\begin{equation}\label{cociclo quater}
f_X(g_1g_2) ~ = ~ f_X(g_1)\cdot g_2 + f_X(g_2).
\end{equation}
Note that $f_X(1)(v) = 0$ for every $v\in V$. Indeed $f_X(1)(v) = w^*_X(v,1) = w_X(v,1)$ and $w_X(v,1) =  0$ by \eqref{pre cociclo}. Hence $f_X(1) = 0$, namely $f_X$ is a $1$-cochain of $G$ over $L(V,W)$. Formula \eqref{cociclo quater} now makes it clear that $f_X$ is a $1$-cocycle of $G$ over $L(V,W)$, i.e. $f_X\in Z^1(L(V,W), G).$

Let now $\alpha\in W^*$ be a linear functional of $W$ and, with $f_X$ as above, for every $g\in G$ put
\begin{equation}\label{alphaX}
\alpha_X(g) ~ :=~   \alpha\circ f_X(g) =  \alpha\circ w_X(\blank , g^{-1}).
\end{equation}
So, $\alpha_X(g)$ is the linear functional of $V^*$ which maps every $v\in V$ onto
$\alpha(w_X(v,g^{-1}))$. As $f_X\in Z^1(L(V,W), G)$, the map $\alpha_X$ is a $1$-cocycle of $G$ over the dual $V^*$ of $V$, i.e. $\alpha_X\in Z^1(V^*, G)$. Explicitly, $\alpha_X(1) = 0$ and
\[\alpha_X(g_1g_2) ~= ~ \alpha_X(g_1)\cdot g_2 +\alpha_X(g_2), ~~~ \forall g_1, g_2 \in G.\]

\subsubsection{Central $1$-extensions of $V$ and the group $H^1(V^*,G)$}

We are now ready to prove the main theorem of this subsection. In order to state it properly, we need some conventions and a definition.
Given a $G$-module $V$, every central $1$-extension of $V$ can be realized in $\KK\times V$ by choosing a function $\phi: V\times G\rightarrow \KK$ satisfying the following property
\begin{equation}\label{Phi(v,g)}
\phi(v, g_1g_2) ~ = ~ \phi(v, g_1) + \phi(vg_1, g_2), ~~~ \forall v\in V, ~ \forall g_1, g_2\in G
\end{equation}
(which is nothing but \eqref{wXgv} where $w_X$ is now called $\phi$) and defining the action of $G$ on $\KK\times V$ as follows:
\begin{equation}\label{1-ext-1}
g ~ : ~ (t,v)\in \KK\times V ~ \rightarrow ~ (t+\phi(v,g), v\cdot g), ~~~ \forall g\in G, ~ \forall v\in V.
\end{equation}
Note that $\phi(v,1) = 0$ for every $v\in V$, as one can see but putting $g_1 = 1$ in \eqref{Phi(v,g)}.

We denote by $U_\phi$ the extension of $V$ realized in this way. In the theorem we are going to state only $1$-extensions of $V$ realized on $\KK\times V$ as explained above are considered.

Given two extensions $U = W\cdot V$ and $U' = W'\cdot V$ of the same $G$-module $V$, an  \emph{isomorphism} from the extension $U$ to the extension $U'$ is just an isomorphism of $G$-modules from $U$ to $U'$ which maps the kernel $W$ of $U$ onto the kernel $W'$ of $U'$.  However, this notion of isomorphism is unsuitable for the next theorem. The notion we need in it is the following one.

Let $U = W\cdot V$ and $U' =  W'\cdot V$ be two extensions of $V$ with the same underlying vector-space structure, namely $U = U'$ and $W = W'$ as vector spaces. We say that an isomorphism of extensions from $U$ to $U'$ is  \emph{rigid} (with respect to the given underlying vector-space structure) if it induces the identity mapping on both $W$ and $U/W$.

Explicitly, if $U = W\cdot V$ and $U' = W'\cdot V$ have the same vector space structure, say  $U = U' = W\times X$ as vector spaces for a given copy $X$ of the $G$-module $V$, and $\psi:U\rightarrow U'$ is a rigid isomorphism from $U$ to $U'$, then there exists a linear mapping $\lambda:X\rightarrow W$ such that $\psi$ maps every vector $(w,v)\in W\times X$ onto $(w+\lambda(v), v)$. Of course, for $\psi$ to be an isomorphism of $G$-modules the mapping $\lambda$ must satisfy the following condition:
\begin{equation}\label{rigid}
\lambda(v\cdot g) - \lambda(v)\cdot g ~ = ~ w'_X(v, g) - w_X(v, g),~~~ \forall g\in G, ~ \forall v\in V,
\end{equation}
with $w_X(v,g)$ and $w'_X(v,g)$ defined according to \eqref{wX} in $U$ and $U'$ respectively.

\begin{theorem}\label{thm link}
  The rigid isomorphism classes of the central $1$-extensions of $V$ bijectively correspond to the elements of $H^1(V^*,G)$. In particular, all split $1$-extensions of $V$ are pairwise rigidly isomorphic and their
rigid isomorphism class corresponds to the null element of $H^1(V^*,G)$.
\end{theorem}
\begin{proof}
Let $U_\phi$ be the (central) $1$-extension of $V$ defined by a mapping $\phi$ as previously explained and let $f_\phi:G\rightarrow V^*$ be the function which maps every $g\in G$ onto the linear functional $f_\phi(g) : v\in V \rightarrow \phi(v, g^{-1})$. Property~\eqref{Phi(v,g)} on $\phi$ implies that $\phi(v,1) = 0$ for every $v\in V$. Hence $f_\phi(1) = 0$ and $f_\phi$ is a $1$-cocycle of $G$ over $V^*$, namely $f_\phi(1) = 0$ and
\begin{equation}\label{basic formula}
f_\phi(g_1g_2) ~ = ~ f_\phi(g_1)\cdot g_2 + f_\phi(g_2), ~~~ \forall g_1, g_2\in G.
\end{equation}
Conversely, for every cocycle $f\in Z^1(V^*, G)$ the mapping $\phi_f:V\times G\rightarrow \KK$ defined by the clause $\phi_f(v, g) := f(g^{-1})(v)$ yields a central $1$-extension $U_f$ of $V$. Of course, $U_{f_\phi} = U_\phi$ and $U_{\phi_f} = U_f$.

The extension $U_\phi$ splits if and only if there exists a linear functional $\lambda\in V^*$ such that the subspace $V_\lambda = \{(\lambda(v), v)\}_{v\in V}$ is stabilized by $G$ in its action on $\KK\times V$ defined as in \eqref{1-ext-1}. This amounts to $\lambda(v) + \phi(v,g) = \lambda(v\cdot g)$, which in turn is equivalent to $f_\phi(g) = \lambda\cdot g - \lambda$. So, $U_\phi$ splits  if and only if $f_\phi$ is $1$-coboundary. The above also shows that, conversely, if $f\in B^1(V^*, G)$ is a $1$-coboundary then $U_f$ splits.

So far we have proved that the central $1$-extensions defined on $\KK\times V$ by a clause like~\eqref{1-ext-1} bijectively correspond to the elements of $Z^1(V^*,G)$, the split ones corresponding to the elements of $B^1(V^*,G)$. In order to finish the proof it remains to prove that, with $U = U_\phi$, $U' = U_{\phi'}$, $f = f_\phi$ and $f' = f_{\phi'}$, the extensions $U$ and $U'$ are rigidly isomorphic if and only if $f'-f \in B^1(V^*,G)$.

Let $\psi$ be a rigid isomorphism from $U$ to $U'$. Then there exists a linear functional $\lambda\in V^*$ such that $\psi$ maps $(t,v)\in U$ onto $(t+\lambda(v), v)\in U'$. In the present setting, Condition~\eqref{rigid} amounts to the following
\begin{equation}\label{iso2}
\phi(v,g) + \lambda(v\cdot g)  ~ = ~ \lambda(v) + \phi'(v,g), ~~~ \forall v\in V, ~ \forall g\in G.
\end{equation}
Hence $\phi'(v,g) - \phi(v,g) = \lambda(v\cdot g) - \lambda(v)$. Equivalently, $f'(g^{-1})(v) - f(g^{-1})(v) = (\lambda\cdot g^{-1})(v) - \lambda(v)$ for every $v\in V$ and every $g\in G$. In short, $f' - f$ is the same as the coboundary $g \rightarrow \lambda\cdot g - \lambda$. Hence $f$ and $f'$ represent the same element of $H^1(V^*,G)$.

Conversely let $f' = f + \beta$ for $\beta \in B^1(V^*,G)$, say $\beta(g) = \lambda - \lambda\cdot g$ for every $g\in G$ and a suitable $\lambda\in V^*$. Then $U$ and $U'$ are rigidly isomorphic, a rigid isomorphism $\psi:U\rightarrow U'$ being defined as follows: $\psi(t,v) = (t- \lambda(v), v)$.
\end{proof}

\begin{remark}
\em
As noticed in Remark \ref{rem1}, formula (2.5) of \cite{SV} and the formula which immediately follows from it are incorrect. Indeed in those two formulas Smith and V\"{o}lklein inadvertedly put $f_\phi(g)(v)$ equal to $\phi(v, g)$ instead of $\phi(v, g^{-1})$.
\end{remark}

Keeping the assumption that $U$ and $U'$ have the same vector-space structure,  \emph{semi-rigid} isomorphisms from $U$ to $U'$ can also be considered. We define them by dropping the requirement to induce the identity on $W$ but keeping the hypothesis that the identity is induced on $U/W$.

Recall that $H^1(V^*,G)$ is a vector space over the same field $\KK$ as $V$.  So, we can consider the projective geometry $\PG(H^1(V^*, G))$ of the linear subspaces of $H^1(V^*, G)$. Keeping the setting implicit in Theorem \ref{thm link} for central $1$-extensions of $V$, consider semi-rigid isomorphisms of central $1$-extensions of $V$. In this setting, a semi-rigid isomorphism is the composition of a rigid isomorphism with a rescaling $t \rightarrow kt$ of the $1$-dimensional vector space $\KK$. Up to a few minor modifications, the proof of Theorem \ref{thm link} also yields the following.

\begin{corollary}\label{cor link}
The semi-rigid isomorphism classes of the non-split central $1$-extensions of $V$ bijectively correspond to the points of $\PG(H^1(V^*,G))$.
\end{corollary}

\subsubsection{Central extensions of rigid $G$-modules}\label{M*sec}

We say that a $G$-module $V$ is {\em rigid} if its automorphism group is the center of the group $\mathrm{GL}(V)$ of all invertible linear transformations of $V$; explicitly, if $\psi\in \mathrm{GL}(V)$ centralizes $G$ then $\psi = k\cdot\mathrm{id}_V$ for some $k\in\KK\setminus\{0\}$. For instance,

\begin{claim}\label{rigid res}
The adjoint $\mathrm{SL}(n+1,\KK)$-module is rigid.
\end{claim}
\begin{proof}
This statement is implicit in Taussky and Zassenhaus \cite[Theorem 2]{TZ} but it can also be proved it in a straightforward way. Let $G = \mathrm{SL}(n+1,\KK)$ and $A = \mathfrak{sl}(n+1,\KK)$ and let $L := \mathrm{GL}(A)$ be the group of all vector space automorphisms of $A$. Obviously $G$, in its adjoint action on $A$, is a subgroup of $L$. The automorphism group of the $G$-module $A$ is the centralizer $C_{L}(G)$ of $G$ in $L$. All we have to show is that $C_L(G)$ is equal to the center $Z(L)$ of $L$. Proving this claim is a routine exercise, which however is a bit laborious. In order to speed up calculations, it is convenient to choose a set $U$ of generators of $G$ that are easy to handle and check that $C_L(U) = Z(L)$. For instance, we can choose the union of a standard complete family of root subgroups; explicitly, $U = \{I+te_{k,h}~|~ t\in \KK, ~ k, h\in \{1,2,..., n+1\}, ~ k\neq h\}$ where $e_{k,h}$ is the square matrix of order $n+1$ with all null entries but the $(k,h)$-entry, which is equal to $1$. The calculations to perform in order to check that with this choice of $U$ indeed $C_L(U) = Z(L)$, are left  to the reader.
\end{proof}

\begin{lemma}\label{nuovo1}
Suppose $V$ is a given rigid $G$-module and $U = W\cdot V$ is a $1$-non-split central extension of $V$ with  $W$ its kernel.
Then $U/H_1\not\cong U/H_2$ (as $G$-modules) for any choice of distinct hyperplanes $H_1, H_2$ of $W$.
\end{lemma}
\begin{proof}
  Up to replacing $U$ with $U/(H_1\cap H_2)$, we can assume with no loss
  of generality that $\dim(W) = 2$. So, there exists a basis $\{w_1, w_2\}$ of $W$ such that $H_i = \langle w_i\rangle$ for $i = 1, 2$. With the notation of \eqref{wXgv},  given a complement $X$ of $W$ in $U$, for every $g\in G$ there exists linear functionals $\phi_1(\blank, g), \phi_2(\blank, g) \in X^*$ such that $w_X(v,g) = \phi_1(v, g)w_1 + \phi_2(v, g)w_2$ for every $v\in X$. Recall that $X$ can also be equipped with a $G$-module structure, where the action $g_X$ of $g\in G$ on $X$ is defined as in \eqref{gX}.

Suppose by contradiction that $U/H_1\cong U/H_2$ and let $\psi$ be an isomorphism from $U/H_1$ to $U/H_2$. Then there exist an automorphism $\psi_X$ of $X$ (with $X$ regarded as a $G$-module, as explained above) a scalar $k \in \KK\setminus\{0\}$ and a linear functional $\lambda\in X^*$ such that $\psi(tw_1, v) = ((kt+\lambda(v))w_2, \psi_X(v))$ for every $(t,v)\in \KK\times X$.   As $\psi$ is an isomorphism of $G$-modules and $G$ acts trivially on $W$, for every $x\in X$ and $g\in G$ we have
\[((\phi_2(\psi_X(x), g) +  \lambda(x))w_2, \psi_X(x)\cdot g_X) ~ = ~ ((k\phi_1(x,g) + \lambda(x\cdot g_X))w_2, \psi_X(x\cdot g_X)),\]
namely
\begin{equation}\label{eq mia1}
\phi_2(\psi_X(x), g) +  \lambda(x)  ~ = ~ k\phi_1(x, g) + \lambda(x\cdot g_X).
\end{equation}
However $V$ is central by assumption and $X$ is a copy of the $G$-module $V$, which is rigid by assumption. Hence $\psi_X = k'\cdot\mathrm{id}_X$ for some $k'\in\KK\setminus\{0\}$. Consequently $\phi_2(\psi_X(x),g) = \phi_2(k'x,g) = k'\phi_2(x,g)$ and \eqref{eq mia1} amounts to the following:

\begin{equation}\label{eq mia2}
\phi_2(x, g) ~ = ~ \frac{k}{k'}\cdot\left(\phi_1(x, g) + \frac{1}{k}\lambda(x\cdot g_X) - \frac{1}{k}\lambda(x)\right).
\end{equation}
Put $\mu := k^{-1}\lambda$ and $w'_2 := k'^{-1}kw_2$. From \eqref{eq mia2} we obtain
\begin{equation}\label{eq mia3}
w_X(x, g) ~= ~ \phi_1(x, g)(w_1+w'_2) + (\mu(x\cdot g_X) - \mu(x))w'_2, ~~~ \forall x\in X, ~\forall g\in G.
\end{equation}
Let now $H_3$ be the hyperplane of $W$ containing $w_1+w'_2$ (unique since $\dim(W) = 2$ by assumption). In view of \eqref{eq mia3}, the extension $U/H_3$ can be described as the direct sum $\KK\oplus X$ with $G$ acting on it as follows:
\begin{equation}\label{eq mia4}
g ~ : ~ t\oplus x\in \KK\oplus X ~ \longrightarrow ~ (t+\lambda(x\cdot g_X)-\lambda(x))\oplus x\cdot g_X,
\end{equation}
for every $g\in G$. It is clear from \eqref{eq mia4} that the subspace $X' := \{\lambda(x)\oplus x\}_{x\in X}$ of $U/H_3 = \KK\oplus X$ is stabilized by $G$. As $X'\cap \KK = \{0\}$, the $1$-extension $U/H_3$ of $V$ splits over $X'$. This contradicts the assumption that $U$ is $1$-non-split.
\end{proof}

\begin{remark}
\em
As a byproduct of \eqref{eq mia2}, two central $1$-extensions of a given rigid $G$-module are isomorphic (if and) only if they are semi-rigidly isomorphic.
\end{remark}

\begin{theorem}\label{M*}
Let $U = W\cdot V$ be a central extension of a rigid $G$-module $V$. Suppose that $U$ is both $1$-non-split and $1$-universal. Then all non-split $1$-extensions of $V$ are central and $H^1(V^*,G)\cong W^*$.
\end{theorem}
\begin{proof}
As $U$ is $1$-universal, every non-split $1$-extension of $V$ can be obtained (up to isomorphisms) by factorizing $U$ over a hyperplane of $W$. Since $U$ is central by assumption, all non-split $1$-extensions of $V$ are central and, since $U$ is also $1$-non-split, $U/H$ is a non-split $1$-extension of $V$ for every hyperplane $H$ of $W$. Moreover, by Lemma \ref{nuovo1} no two distinct hyperplanes of $W$ give rise to isomorphic $1$-extensions of $V$. It follows that the hyperplanes of $W$ bijectively correspond to the isomorphism classes of the non-split $1$-extensions of $V$. By Corollary \ref{cor link}, the hyperplanes of $W$ bijectively correspond to the points of $\PG(H^1(V^*,G))$.

We still must show that the above bijection between the set of hyperplanes of $W$ and the set of points of  $\PG(H^1(V^*,G))$ is induced by an isomorphism from $W^*$ to $H^1(V^*,G)$. The isomorphism we look for is provided by equation \eqref{alphaX}, where for every $\alpha\in W^*\setminus\{0\}$ a $1$-cocycle $\alpha_X\in Z^1(V^*,G)$ is defined which yields a $1$-extension of $V$ isomorphic to the extension obtained by factorizing $U$ over the kernel of $\alpha$. Since the hyperplanes of $W$ correspond to the points of $\PG(H^1(V^*,G))$, we are sure that $\alpha_X$ is not a coboundary. Replacing $\alpha$ with $k\alpha$ for a scalar $k\in \KK\setminus\{0\}$ amounts to replace $\alpha_X$ with $k\alpha_X$, hence the class $[\alpha_X]\in H^1(V^*,G)$ with $k[\alpha_X]$. (Note also that, if the $1$-extension defined by $\alpha_X$ is realized on $\KK\times V$ as implicitly assumed in Corollary \ref{cor link}, when we replace $\alpha_X$ with $k\alpha_X$ we apply a semi-rigid automorphism to that extension).  Clearly, if $\alpha$ and $\beta$ are linear functionals in $W^*$ and $\gamma = \alpha+\beta$ then $\gamma_X = \alpha_X + \beta_X$.
\end{proof}

\subsection{Projective embeddings and extensions}

We remind the reader that, according to the definition of projective embedding, if $\varepsilon:\Gamma\to\PG(V)$ is a projective embedding of a point-line geometry $\Gamma$ in the projective space $\PG(V)$ of a vector space $V$, the image $\eps(\Gamma)$ of $\Gamma$ by $\eps$ spans $\PG(V)$.

In many (but not all) cases, if the geometry $\Gamma$ admits a (possibly non faithful) action $G_\Gamma$ of a group $G$ as a subgroup of its automorphism group and $\eps\colon \Gamma\rightarrow\PG(V)$ is a $G_\Gamma$-homogeneous projective embedding of $\Gamma$, then the vector space $V$ is naturally endowed with a structure of $G$-module in such a way that the projective action $\mathrm{P}(G)$ of $G$ on $\PG(V)$ is just the lifting of $G_\Gamma$ to $\PG(V)$ through $\eps$ (hence $G_\Gamma \cong \mathrm{P}(G)$ and $\mathrm{P}(G)$ stabilizes $\eps(\Gamma)$). If this is the case we say that $G$  \emph{lifts to $V$ through $\eps$}. Henceforth, when dealing with homogeneous embeddings, we shall always implicitly assume that this is indeed the case.

Given a $G$-module $V$ and a point-line geometry $\Gamma$, if $\Gamma$ admits an embedding $\eps:\Gamma\rightarrow \PG(V)$ such that $\mathrm{P}(G)$ stabilizes $\eps(\Gamma)$, then we say that the module $V$  \emph{hosts} the embedding $\eps$.

In the previous paragraphs we distinguish between $G$ and $G_\Gamma$ or $\mathrm{P}(G)$ but henceforth we will freely omit to do so, provided that this abuse will not cause any misunderstanding.

Our aim in this subsection is to determine conditions which ensure that, given a $G$-module $V$ hosting an embedding $\eps$ of a geometry $\Gamma$ and an extension $U$ of $V$, the $G$-module $U$ hosts a cover of $\eps$, possibly the relatively universal cover of $\eps$.

Throughout this subsection $\Gamma$ is a given (connected) point-line geometry, $G$ a group acting (possibly non-faithfully) on $\Gamma$ as a group of automorphisms and $\eps:\Gamma\rightarrow \PG(V)$ is a given projective embedding of $\Gamma$. We assume that $G$ lifts to $V$ through $\eps$.  Accordingly, $V$ is a $G$-module.

\begin{definition}\label{Gcover}
\em
With $\Gamma$, $G$, $V$ and $\eps$ as above, let $\eps':\Gamma\rightarrow\PG(V')$ be a cover of $\eps$ and $\psi:V'\rightarrow V$ the projection of $\eps'$ onto $\eps$. We say that $\eps'$ is a $G$-\emph{cover} of $\eps$ if $G$ also lifts to $V'$ through $\eps'$ and $\psi$ is a morphism of $G$-modules from $V'$ to $V$, namely the kernel $W$ of $\psi$ is stabilized by $G$ and $\psi$ induces an isomorphism of $G$-modules from $V'/W$ to $V$. (So, $V'$ is an extension of the $G$-module $V$ with $W$ as its kernel).
\end{definition}

For instance, let $\teps:\Gamma\rightarrow\PG(\widehat{V})$ be the relatively universal cover of $\eps$. Then $\teps$ is a $G$-cover of $\eps$ (Appendix \ref{appendix A}, Proposition \ref{propA}). If moreover $\widehat{V}$ is  central as an extension of $V$ then all covers of $\eps$ are $G$-covers.

\begin{remark}
\em
We need $\psi$ to be a morphism of $G$-modules for, otherwise, we could hardly exploit results on module extensions in the investigation of covers of projective embeddings. However one might wonder if it is really necessary to assume this property openly in Definition \ref{Gcover}. Is it not possible to obtain it from a seemingly weaker but more natural property? For instance, assume only that $\psi$ maps the action of $G$ on $\eps'(\Gamma)$ onto the action of $G$ on $\eps(\Gamma)$, namely $\psi(\eps'(p)\cdot g) = \psi(\eps'(p))\cdot g$ for every point $p$ of $\Gamma$ an every $g\in G$. Then, since $\eps'(\Gamma)$ spans $\PG(V')$, the kernel $W$ of $\psi$ is stabilized by $G$ and $\psi$ maps the action of $G$ on $\PG(V'/W)$ onto the action of $G$ on $\PG(V)$. For certain choices of $G$ (when $G$ is perfect, for instance) this forces $\psi$ to induce an isomorphism of $G$-modules from $V'/W$ to $V$, but we do not know if this is always the case. \\
 \end{remark}

\begin{lemma} \label{lemma 1.4}
Let $\eps':\Gamma\rightarrow\PG(V')$ be a $G$-cover of $\eps$ and suppose that $V'$ is central as an extension of the $G$-module $V$.  Suppose moreover that $G$ acts transitively on the set of points of $\Gamma$. Then $V'$ is a totally non-split extension of $V$.
\end{lemma}
\begin{proof}
By way of contradiction suppose that $V'$ splits, i.e. $V'=W\oplus X$ where $X$ is a complement of $W\neq \{0\}$ stabilized by $G.$
Let $v_0\in V'\setminus W$ be a representative vector of a point $\eps'(p_0)\in \eps'(\Gamma).$ So $v_0=w_0+x_0$ for suitable vectors $w_0\in W$ and $x_0\in X\setminus\{0\}$. As $G$ acts point-transitively on $\Gamma$, every other point of $\eps'(\Gamma)$ can be obtained as $g(\eps'(p_0))$ for some $g\in G$, so it is represented in $V'$ by $v_0\cdot g = w_0g+x_0g=w_0+x_0g$, since $G$ acts trivially on $W.$ So, for every point $p\in \Gamma$, the point $\eps'(p)$  can be represented by a vector of the form  $w_0+x_p$ where $w_0\in W$ does not depend on $p$ and $x_p\in X\setminus \{0\}.$

Let now $p_1$ and $p_2$ be two collinear points in $\Gamma$. The vector $v := x_{p_1}-x_{p_2}=(w_0+x_{p_1})-(w_0+x_{p_2})$ represents a point of the line of $\PG(V')$ through $\eps'(p_1)$ and $\eps'(p_2).$ That line is a line of $\eps'(\Gamma)$, since $p_1$ and $p_2$ are collinear in $\Gamma$. Hence there exists a point $p_3$ in the line of $\Gamma$ through $p_1$ and $p_2$ such that $\langle v\rangle = \eps'(p_3)$. However $v\in X$, since $v = x_{p_1}-x_{p_2}$ and $x_{p_1}, x_{p_2}\in X$. So $\eps'(p_3)$ is represented by a vector $v \in X$. On the other hand, $\eps'(p_3)$ is represented by $w_0 + x_{p_3}$. Therefore $w_0 = 0$ and $x_{p_3}$ is proportional to $v$. So, every point of $\eps'(\Gamma)$ is represented by a vector in $X$. This implies that $\eps'(\Gamma)$ generates $\PG(X)$ which is a proper subspace of $\PG(V')$. A contradiction has been reached.

Therefore $V'$ is non-split. Since $V'/W_0$ is a central extension of $V$, the same argument used for $V'$ applies to $V'/W_0$. Hence $V'/W_0$ is non-split for any arbitrary proper subspace $W_0$ of $W.$ So, $V'$ is totally non-split.
\end{proof}

In the sequel $V' = W\cdot V$ is a given extension of $V$ and $\pi$ is the morphism from $\PG(V')$ to $\PG(V)$ induced by the natural projection of $V'$ onto $V$. For the moment we assume neither that $V'$ hosts a cover of $\eps$ nor that it is central as an extension of $V$. We also make no assumptions on $G.$

Given a point $p'\in \PG(V')\setminus\PG(W)$, let $p := \pi(p')$. Let $G_{p'}$ and $G_p$ be the stabilizers in $G$ of $p'$ and  $p$ in the action of $G$ on respectively $\PG(V')$ and $\PG(V)$. Clearly, $G_{p'}\subseteq G_p$ but in general $G_{p'} \subsetneq G_p$.

\begin{definition}
\em
If $G_{p'} = G_p$ then we say that $p'$ is  \emph{well-stabilized} (\emph{in $G$}). Similarly, given a line $\ell'$ of $\PG(V')$ skew with $\PG(W)$ let $\ell = \pi(\ell')$. We say that $\ell'$ is  \emph{well-stabilized} if $G_{\ell'} = G_{\ell}$.
\end{definition}

The group $G$ stabilizes the set of well-stabilized points (lines) of $\PG(V')$. Indeed if $x$ is a point or a line of $\PG(V')$ then $G_{x\cdot g} = g^{-1}G_{x}g$. Assuming that $x$ is exterior to $\PG(W)$, the same holds for $\pi(x)$. Note also that $\pi(x\cdot g) = \pi(x)\cdot g$. Hence $G_{x\cdot g} = G_{\pi(x)\cdot g}$ if and only if $G_{x} = G_{\pi(x)}$.

\begin{lemma}\label{well stab}
Let $O'$ be an orbit of $G$ on the set of well-stabilized points (lines) of $\PG(V')$. Then $\pi$ bijectively maps $O'$ onto an orbit $O$ of $G$ on the set of points (lines) of $\PG(V)$.
\end{lemma}
\begin{proof}
Since $\pi(x\cdot g) = \pi(x)\cdot g$ for every point or line $x$ of $\PG(V')$ exterior to $\PG(W)$, the projection $\pi$ maps $O'$ onto an orbit $O$ of $G$. Let $x, y\in O'$. Then $G_x = G_{\pi(x)}$ and $G_y = G_{\pi(y)}$. Moreover $y = x\cdot g$ for some $g\in G$. Accordingly, $\pi(y) =\pi(x)\cdot g$. If $\pi(x) = \pi(y)$ then $g\in G_{\pi(x)}$. Hence $g\in G_x$ since $x$ is well-stabilized. Consequently $y = x$.
\end{proof}

\begin{definition}
\em
We  say that a point (a line) of $\PG(V')\setminus\PG(W)$ is an  \emph{$\eps$-point} (an  \emph{$\eps$-line}) if $\pi$ maps it onto a point (a line) of $\eps(\Gamma)$. Since the group $G$ stabilizes $\eps(\Gamma)$ in its action on $\PG(V)$, it also stabilizes the set of $\eps$-points and the set of $\eps$-lines in its action on $\PG(V')$.
\end{definition}

\begin{definition}\label{very well}
\em
We say that a well-stabilized  $\eps$-point $p'\in\PG(V')$ is \emph{very well-stabilized} (in $G$) if the set of $\eps$-lines through $p'$ contains a subset $L(p')$ of well-stabilized $\eps$-lines such that $\pi$ induces a bijection from $L(p')$ to the set of lines of $\eps(\Gamma)$ through the point $\pi(p')$ of $\eps(\Gamma)$ and the stabilizer of $p'$ in $G$ stabilizes $L(p')$.
\end{definition}

Clearly, $G$ stabilizes the set of well-stabilized $\eps$-points of $\PG(V')$.

\begin{definition}
\em
We say that $G$ is \emph{semi-flag-transitive} on $\Gamma$ if it is transitive on the set of points of $\Gamma$ and for every line $\ell$ of $\Gamma$, the stabilizer of $\ell$ in $G$ is transitive on the set of points of $\ell.$
\end{definition}

For instance, $\mathrm{SL}(n+1,\KK)$ acts semi-flag-transitively but not flag-transitively on $A_{n,\{1,n\}}(\KK)$.

\begin{lemma}\label{lemma 1.5}
Suppose that $G$ acts semi-flag-transitively on $\Gamma$ and at least one very well-stabilized $\eps$-point exists in $\PG(V')$. Then the $G$-module $V'$ admits a submodule $V''$ satisfying both the following:
\begin{enumerate}
\item  $W+V'' = V'$ (hence $V''$ is an extension of $V$ with $W\cap V''$ as its kernel);
\item $V''$ hosts a $G$-cover $\eps''$ of $\eps$ and the restriction of $\pi$ to $V''$ yields the morphism from the embedding $\eps''$ to the embedding $\eps$.
\end{enumerate}
\end{lemma}
\begin{proof}
Let $p'_0$ be a very well-stabilized $\eps$-point of $\PG(V')$ and $L(p'_0)$ a set of lines through $p'_0$ as in Definition \ref{very well}. Let $O' = \{p'_0\cdot g\}_{g\in G}$ be the orbit of $G$ containing $p'_0$. By the semi-flag-transitivity of $G$ on $\Gamma$ we have that $\ell'\subseteq O'$ for every line $\ell'\in L(p'_0)$. Hence $\ell'\cdot g \subseteq O'$ for every $\ell'\in L(p'_0)$ and every $g\in G$.

Let $\Gamma'$ be the subgeometry of $\PG(V')$ with $O'$ as the set of points and $\{\ell'\cdot g ~:~ \ell'\in L(p'_0),~ g\in G\}$ as the set of lines.  By Lemma \ref{well stab} the projection $\pi$ induces an isomorphism from $\Gamma'$ to $\eps(\Gamma)$. Accordingly, the mapping $\eps''$ which maps every point $p$ of $\Gamma$ onto $\pi^{-1}(\eps(p))\cap O'$ is a projective embedding of $\Gamma$ in the span $\langle O'\rangle$ of $O'$ in $\PG(V')$. This embedding is clearly a $G$-cover of $\eps$.

Let $V''$ be the subspace of $V'$ corresponding to $\langle O'\rangle$. Clearly $V''$ is stabilized by $G$ and, since $\eps(O')$ (which is the point-set of $\eps(\Gamma)$) spans $\PG(V)$, we also have that  $W+V'' = V'$.
\end{proof}

\begin{theorem}\label{theo 1.5}
Suppose that $G$ acts semi-flag-transitively on $\Gamma$ and at least one very well-stabilized $\eps$-point exists in $\PG(V')$.  Assume also that $V'$ is central as an extension of $V$. The following are equivalent.
\begin{enumerate}
\item The vector space $V'$ hosts a $G$-cover $\eps'$ of $\eps$ and $\pi$ is the morphism from $\eps'$ to $\eps$.
\item The extension $V' = W\cdot V$ is totally non-split.
\item The extension $V'$ is $1$-non-split.
\end{enumerate}
\end{theorem}
\begin{proof}
  The equivalence of claims 2 and 3 has been stated in Proposition \ref{1-ext}. With $V''$ as in the statement of Lemma \ref{lemma 1.5}, the equivalence of 1 and 2 amounts to the following: $V'' = V'$ if and only if $V'$ is totally non-split. The `only if' part of this claim follows from Lemma \ref{lemma 1.4}. Turning to the `if' part, suppose by contradiction that $V'' \subsetneq V'$ and
  put $W' = V''\cap W$. Then $V'/W' = W/W'\oplus V''/W'$. However $V''/W' \cong V$, as stated in claim 1 of Lemma \ref{lemma 1.5}. So, at least one proper quotient of $V'$ splits, against the hypotheses made on $V'$.
\end{proof}

\begin{theorem}\label{1-universal bis}
Suppose that $G$ is perfect and $V$ is rigid. Then every $1$-universal $1$-non-split central extension of $V$ is $1$-complete.
\end{theorem}
\begin{proof}
Let $\overline{V} = W\cdot V$ be a $1$-universal $1$-non-split central extension of $V$. By way of contradiction suppose $\widetilde{V} = K\cdot \overline{V}=(K\cdot W)\cdot V$ is a non-split extension of $\overline{V}$ where $\dim(K)=1$. By assumption, $G$ acts trivially on $W=(K\cdot W)/K$ and since it is perfect (perfect groups have trivial $1$-dimensional representations), $G$ also acts trivially on $K$. So, given a basis $\{k_0\}$ of $K$, the action of $G$ on  $K\times W$ is as follows:  $g\in G$ maps $(k,w)\in K\times W$ onto $(k+k_g(w)k_0, w)$ where $k_g\in W^*$ is a linear functional of $W$. Accordingly, $g_1g_2$ maps $(k,w)$ onto $(k+ k_{g_1g_2}(w)k_0, w)$. However $(k,w)\cdot g_1g_2 = ((k,w)\cdot g_1)\cdot g_2 =  (k+k_{g_1}(w)k_0+k_{g_2}(w)k_0, w)$. Therefore $k_{g_1g_2} = k_{g_1}+k_{g_2}$ and the action of $G$ on $K\times W$ is commutative. Since $G$ is perfect, this action is necessarily the trivial one, hence $k_g$ is always the null functional.

We shall now prove that $\widetilde{V} = (K\times W)\cdot V$ is a $1$-non-split extension of $V$. Now it is convenient to regard $K\times W$ as a direct sum $K\oplus W$ rather than a product: in this way $K$ and $W$ can be regarded as vector subspaces of $\widetilde{V}$.

 Suppose for a contradiction that a hyperplane $W'$ of $K\oplus W$ exists such that $\widetilde{V}/W'$ splits. Suppose firstly that $W'$ contains $K$. Then the kernel $(W+K)/K \cong W$ of the extension $(\overline{V}+K)/K$ of $V$ admits a hyperplane $W'/K$ such that the quotient of $(\overline{V}+K)/K$ over $W'/K$ splits. However $(\overline{V}+K)/K \cong \overline{V}$, which is $1$-non-split by assumption. Therefore $W'$ cannot contain $K$. Hence $K\oplus W = W'\oplus K$. We are assuming that $\widetilde{V}/W' = (W'+K)/W' \oplus X/W'$ for a subspace $X$ of $\widetilde{V}$ containing $W'$ and stabilized by $G$. Clearly, $X \cong \widetilde{V}/K \cong \overline{V}$ (isomorphisms of $G$-modules). So $\widetilde{V} = K\oplus X$ where $X\cong \overline{V}$ is stabilized by $G$. Consequently, the extension $\widetilde{V} = K\cdot\overline{V}$ splits. This contradicts the assumption that $\widetilde{V}$ is a non-split extension of $\overline{V}$.  Thus, we are forced to conclude that $\widetilde{V}$ is $1$-non-split.

 Consider now the quotient $V'=\widetilde{V}/W=K\cdot(\overline{V}/W)\cong K\cdot V$, which is a non-split central $1$-extension of $V$ (had $V'$
 been split then $\widetilde{V} = K\cdot \overline{V}$ would split as well). However $\overline{V}$ is $1$-universal, by assumption. Consequently, this extension can also be obtained as a quotient of $\overline{V}$ over a suitable hyperplane $H$ of $W$. So, we can obtain $V'$ by factorizing $\widetilde{V}$ over two different hyperplanes of $K\oplus W$, namely $W$ and $K\oplus H$. This is impossible in view of Lemma \ref{nuovo1} since, as proved in the previous paragraph, $\widetilde{V}$ is $1$-non-split and $V$ is rigid by assumption. Ultimately, we have proved that $\widetilde{V}$ cannot exist.
\end{proof}

Let $\widehat{\eps}\colon \Gamma\rightarrow\PG(\widehat{V})$ be the relatively universal cover of $\eps$ and let $W$ be the kernel of the projection $\pi:\widehat{V}\rightarrow V$ of $\teps$ onto $\eps$.

Then, as remarked in the comment following Definition \ref{Gcover}, the embedding $\teps$ is a $G$-cover of $\eps$. Accordingly, $G$ acts on $\widehat{V}$, stabilizes $W$ and $\pi$ induces an isomorphism of $G$-modules from $\widehat{V}/W$ to $V$. So $\widehat{V} = W\cdot V$ is an extension of $V$.

\begin{theorem}\label{1-universal}
Let $\widehat{\eps}\colon \Gamma\rightarrow\PG(\widehat{V})$ be the relatively universal cover of $\eps$ and let $W$ be the kernel of the projection $\pi:\widehat{V}\rightarrow V$ of $\teps$ onto $\eps$. So $\widehat{V} = W\cdot V$ is an extension of $V$.
Suppose that $G$ acts trivially on $W$ and semi-flag-transitively on $\Gamma$. Then both the following hold.
\begin{enumerate}
\item The extension $\widehat{V} = W\cdot V$ is totally non-split.
\item Suppose that every non-split $1$-extension of $V$ contains at least one $\eps$-point very well-stabilized in $G$. Then the extension $\widehat{V}$ is $1$-universal.
\end{enumerate}
\end{theorem}
\begin{proof}
Claim 1 follows from  Lemma~\ref{lemma 1.4}. Turning to claim 2, let $V' = K\cdot V$ be a non-split $1$-extension of $V$. By assumption, $\PG(V')$ contains at least one $\eps$-point very well-stabilized by $G$. Let $V''$ be as in the statement of Lemma \ref{lemma 1.5}. Then $V'' = V'$. Indeed if otherwise  then $V'' \cong V$ and, since $\dim(K) = 1$, claim 2 of Lemma \ref{lemma 1.5} (with $K$ in the role of $W$) forces $V' = K\oplus V \cong K\times V$, contradicting the hypothesis that $V'$ is non-split. So, $V'' = V'$ as claimed. Part 1 of Lemma \ref{lemma 1.5} now implies that $V'$ hosts a cover $\eps'$ of $\eps$. Since $\teps$ is the universal cover of $\eps$, it also covers $\eps'$, namely we obtain $\eps'$ by factorizing $\widehat{V}$ over a subspace (in fact a hyperplane) of $W$. Clearly $\teps$, being relatively universal, is also the relatively universal cover of $\eps'$. Hence it is a $G$-cover of $\eps'$. Accordingly, the $G$-module $V'$ is a quotient of $\widehat{V}$. Claim 2 is proven.
\end{proof}

\begin{remark}
\em
By Theorems \ref{1-universal bis} and \ref{1-universal}, under the hypotheses of claim 2 of Theorem \ref{1-universal}, if $G$ is perfect and $V$ is rigid then $\widehat{V}$ is also $1$-complete. We conjecture that a stronger conclusion can be drawn, namely that under these hypotheses $\widehat{V}$ admits no non-split extensions, but we have done no serious attempts to prove this.
\end{remark}

\begin{corollary}\label{1-universal ter}
Let $\widehat{\eps}\colon \Gamma\rightarrow\PG(\widehat{V})$ be the relatively universal cover of $\eps$ and let $W$ be the kernel of the projection $\pi:\widehat{V}\rightarrow V$ of $\teps$ onto $\eps$. So $\widehat{V} = W\cdot V$ is an extension of $V$.
Suppose that $G$ acts trivially on $W$ and transitively on the set of points of $\Gamma$. Suppose moreover that $G$ is perfect and $V$ is rigid.

Let $\overline{V}$ a $1$-universal and $1$-non-split central extension of $V$ and suppose that this extension hosts a $G$-cover $\bar{\eps}$ of $\eps$. Then $\overline{V} = \widehat{V}$.
\end{corollary}
\begin{proof}
As $\bar{\eps}$ covers $\eps$ and $\teps$ is the relatively universal cover of $\eps$, the embedding $\teps$ also covers $\bar{\eps}$. Accordingly, $\overline{V} = \widehat{V}/K$ for some subspace $K$ of $W$ and $\bar{\eps}$ is a $G$-cover. Suppose by contradiction that $K \neq \{0\}$. Let $H$ be a hyperplane of $K$. Then $\widehat{V}/H = (K/H)\cdot\overline{V}$ is a central $1$-extension of $\overline{V}$ and hosts a $G$-cover of $\bar{\eps}$. By Lemma \ref{lemma 1.5} with $\overline{V}$ and $\bar{\eps}$ in the roles of $V$ and $\eps$ respectively and $\widehat{V}/H$ in the role of $V'$, we obtain that $\widehat{V}/H$ is a non-split $1$-extension of $\overline{V}$. However, under the hypotheses assumed on $G$ and $\overline{V}$, Theorem \ref{1-universal bis} states that $\overline{V}$ admits no non-split $1$-extension. Therefore $K = \{0\}$, namely $\overline{V} = \widehat{V}$.
\end{proof}

 \section{Proof of Lemma \ref{main lemma} and Theorem \ref{main theorem}}\label{Sec 3}

Henceforth $\Gamma := A_{n,\{1,n\}}(\KK)$, $G = \mathrm{SL}(n+1,\KK)$, $A$ is the adjoint module for $G$ described in the Introduction of this paper, $\neps:\Gamma\rightarrow\PG(A)$ and $\tneps:\Gamma\rightarrow\PG(\widehat{A})$ are respectively the natural embedding of $\Gamma$ and its relatively universal cover and $M$ is the kernel of the projection from $\tneps$ to $\neps$ (see Appendix~\ref{Appendix A.1}). We know that $\neps$ is $G$-homogeneous. Since the relatively universal cover of a $G$-homogeneous embedding is a $G$-cover, $\widehat{A} = M\cdot A$ is an extension of the $G$-module $A$ with $M$ as the kernel.

\subsection{Proof of Lemma \ref{main lemma}}\label{sec 3.1}

The next proposition is a generalization of Theorem (1.4) of \cite{SV}.

\begin{proposition}\label{lemma A}
$\widehat{A}=M\cdot A$ is a central extension of $A.$
\end{proposition}
\begin{proof}
Throughout the proof of this lemma we adopt the following notation. Given a subfield $\FF$ of $\KK$ and a set $X$ of vectors of a $\KK$-vector space $V$, we denote by $\langle X\rangle_\FF$ the $\FF$-span of $X$ in $V$, namely the set of all $\FF$-linear combinations of a finite
  number of vectors of $X$. The dimension $\dim_{\FF}(\langle X\rangle_\FF)$ of $\langle X\rangle_\FF$ is its dimension as an $\FF$-vector space. Chosen a basis $B$ of $V$,  we put $V(\FF) := \langle B\rangle_\FF$ and denote by $P(V(\FF))$ the projective geometry of $V(\FF)$, regarded as a subgeometry of $\PG(V)$, every $1$-dimensional subspace of $\langle B\rangle_\FF$ being identified with the $1$-dimensional subspace of $V$ containing it.

With $V = V(n+1, \KK)$, we denote by $\Gamma(\FF)$ the point-hyperplane geometry of $\PG(V(\FF)) \cong \PG(n,\FF)$, regarded as a subgeometry of $A_{n,\{1,n\}}(\KK)$. Put $G(\FF) := \mathrm{SL}(n+1, \FF)$ and let $A(\FF)$ be its adjoint module, these modules being nested in such a way that if $\FF\subseteq \FF'$ then $A(\FF') = \langle A(\FF)\rangle_{\FF'}$. Also $M(\FF):=\widehat{A}(\FF)\cap M$, where $\widehat{A}(\FF)$ is the $\KK$-span in $\widehat{A}$ of the set of vectors which represent the points of the  image $\tneps(\Gamma(\FF))$ of $\Gamma(\FF)$ in $\PG(\widehat{A})$ via $\tneps$.  Clearly, if $\FF \subseteq \FF'$ then  $\widehat{A}(\FF)\subseteq \widehat{A}(\FF')$, hence $M(\FF)\subseteq M(\FF')$.
Note also that, for every subfield $\FF$ of $\KK$, the projection of $\widehat{A}$ onto $A$ along $M$ maps $\widehat{A}(\FF)$ onto $A$. Indeed $\widehat{A}(\FF) \supseteq A(\FF)$ and $\langle A(\FF)\rangle_\KK = A$.

Let $\KK_0$ be the minimal subfield of $\KK$. Put $G_0:= G(\KK_0)$ for short.  By Cooperstein \cite{C98b} the geometry $\Gamma(\KK_0)$ admits a generating set X of size $n^2+2n = \dim(A)$. (Cooperstein mentions only finite prime fields in \cite{C98b}, but the results he obtains in that paper hold as well when $\KK_0$ is the field of rational numbers). Therefore, if we choose a representative vector for every point of $\tneps(X)$ we obtain a set $B_X$ of generators of $\widehat{A}(\KK_0)$. However $\dim(\widehat{A}(\KK_0)) \geq \dim(A) = n^2+2n$. Accordingly, $B_X$ is a basis of $\widehat{A}(\KK_0)$ and $\dim(\widehat{A}(\KK_0)) = \dim(A) =  n^2+2n$. As the projection of $\widehat{A}$ onto $A$ along $M$ maps $\widehat{A}(\KK_0)$ onto $A$, we obtain that $\widehat{A}(\KK_0)$ is a complement of $M$ in $\widehat{A}$; with no loss, we can also assume that $\widehat{A}(\KK_0) = A$. We have proved that $M(\KK_0) = \{0\}$.

For every subfield $\FF$ of $\KK$ the group $G_0$, being a subgroup of $G(\FF)$, stabilizes $M(\FF)$. Let $\Phi$ be the set of all subfields $\FF$ of $\KK$ such that the action of $G_0$ on the module $M(\FF)$ is trivial. The action of $G_0$ on $M(\KK_0)=\{0\}$ is obviously trivial. Hence $\Phi\neq\emptyset$. Let now $(\{\FF_i\}_{i\in I}, \subseteq)$ be a chain of subfields of $\KK$. For any $x\in M(\bigcup_{i\in I}\FF_i)$ there exists $i\in I$ such that $x\in M(\FF_i)$. It follows that $M(\bigcup_{i\in I}\FF_i)=\bigcup_{i\in I} M(\FF_i)$. Therefore, if $G_0$ acts trivially on $M(\FF_i)$ for every $i\in I$ then it also acts trivially on $M(\bigcup_{i\in I}\FF_i)$. Accordingly, every chain in $\Phi$ admits an upper bound in $\Phi$.  By Zorn's lemma $\Phi$ admits a maximal element, say $\KK_1$.

Proving the lemma amounts to prove that $\KK_1=\KK$. Suppose by contradiction that $\KK_1\subsetneq \KK$. Pick an element $\gamma\in\KK\setminus\KK_1$ and let $\KK_2:=\KK_1(\gamma)$ be the simple extension of $\KK_1$ by means of $\gamma$. In view of Blok and Pasini \cite[Corollary 4.8]{BPa01} the geometry $\Gamma(\KK_2)$ is generated by $\Gamma(\KK_1)$ and possibly one extra point $p \in \Gamma(\KK_2)\setminus\Gamma(\KK_1)$. Therefore $\widehat{A}(\KK_1)$ has codimension at most $1$ in $\widehat{A}(\KK _2)$.
However $\widehat{A}(\KK_2) =M(\KK_2)\cdot A$ and
$\widehat{A}(\KK_1) = M(\KK_1)\cdot A.$  It follows that $M(\KK_1)$ has codimension at most $1$ in $M(\KK_2)$. The group $G_0$ acts trivially on $M(\KK_1)$ by the choice of $\KK_1$. Therefore $M(\KK_1)\subsetneq M(\KK_2)$, since $\KK_1$ is maximal in $\Phi$. So, $M(\KK_1)$ has codimension $1$ in $M(\KK_2)$ and, chosen a vector $v\in M(\KK_2)\setminus M(\KK_1)$  , we have $M(\KK_2) = M(\KK_1)\cdot\langle v\rangle_\KK$ (as a $G_0$-module). Since $G_0$ acts trivially on  both $M(\KK_1)$ and $M(\KK_2)/M(\KK_1) \cong \langle v\rangle_\KK$, its action on  $M(\KK_2)$ can be only as follows:
   \begin{equation} \label{trac}
(x, tv)\cdot g = (tf(g)+x, t\lambda(g)v), \hspace{5 mm} \forall g\in G_0,~ \forall t\in \KK , ~\forall  x\in \langle M(\KK_1)\rangle ;
 \end{equation}
where $\lambda$ in a homomorphims from $G_0$ to the multiplicative group of $\KK$ and $f$ is a mapping from $G_0$ to the additive group of $M(\KK_1)$ such that
   \begin{equation} \label{trac bis}
f(g_1g_2) ~ = ~ f(g_1) + \lambda(g_1)f(g_2), \hspace{5 mm} \forall g_1, g_2\in G_0.
 \end{equation}
However $G_0$ is perfect. Hence $\lambda(g) = 1$ for every $g\in G_0$. Accordingly, \eqref{trac bis} boils down to the following:
\[f(g_1g_2) ~ = ~ f(g_1) + f(g_2), \hspace{5 mm} \forall g_1, g_2\in G_0.\]
So, $f$ is a homomorphism from $G_0$ to the additive group of $M(\KK_1)$. The latter is commutative while $G_0$ is perfect. Therefore $f(g) = 0$ for every $g\in G_0$. Formula \eqref{trac} now shows that $G_0$ centralizes $M(\KK_2)$, namely $\KK_2\in \Phi$. This conclusion contradicts the maximality of $\KK_1$ in $\Phi$.

Therefore $\KK=\KK_1\in\Phi$ and the action of $G_0$ on $M=M(\KK)$ is trivial. To conclude, observe that the group
$G:=\mathrm{SL}(n+1,\KK)$ does not admit any proper normal subgroup containing $G_0$. Indeed $G/Z(G)$ is simple.
Hence the action of $G$ on $M$ must also be trivial.
\end{proof}

As noticed in the Introduction, the group $G$ acts semi-flag-transitively on $\Gamma$. Moreover, it lifts to the adjoint module $A$ through $\neps$; hence $\tneps$, being the relatively universal cover of $\neps$, is a $G$-cover of $\neps$. Furthermore $\widehat{A} = M\cdot A$ is a central extension of $A$, as proved in Proposition \ref{lemma A}. Therefore $\widehat{A}$ is a totally non-split extension of $A$, by Lemma \ref{lemma 1.4}. In particular:
\begin{corollary}\label{A1}
The extension $\widehat{A}$ is $1$-non-split.
 \end{corollary}

Before proceeding further, we need to recall a construction of the dual $A^*$ of $A$ as a quotient of the space $M_{n+1}(\KK)$ of all square matrices of order $n+1$ with entries in $\KK$ and give some information on the isomorphism $H^1(A^*,G)\cong \Der(\KK)$.

The dual module $A^*$ of $A$ is isomorphic in a standard way to the $G$-module $M_{n+1}(\KK)/\langle I\rangle$, where $\langle I\rangle$ is the space of all scalar matrices, namely the $1$-dimensional subspace of $M_{n+1}(\KK)$ generated by the identity matrix $I$. For every $\alpha \in A^*$ there exists a unique coset $\mu(\alpha) \in M_{n+1}(\KK)/\langle I\rangle$ such that
\begin{equation}\label{Astar-M/I}
\alpha(a) ~ = ~ \mathrm{Tr}(b\cdot a), \hspace{5 mm} \forall b\in \mu(\alpha), ~ \forall a\in A.
\end{equation}
(Of course, $b\cdot a$ is the usual row-by-column product of $b$ and $a$ and $\mathrm{Tr}(b\cdot a)$ is the trace of the matrix $b\cdot a$). The mapping $\mu$ defined as above is the standard isomorphism from $A^*$ to $M_{n+1}(\KK)/\langle I\rangle$. So, we may take the liberty of regarding $A^*$ and $M_{n+1}(\KK)/\langle I\rangle$ as the same objects, thus restating~\eqref{Astar-M/I} as follows:
\begin{equation}\label{Astar-M/I bis}
\alpha(a) ~ = ~ \mathrm{Tr}(\rho(\alpha)\cdot a), \hspace{5 mm} \forall \alpha\in A^*, ~ \forall a\in A,
\end{equation}
where $\rho(\alpha)$ stands for a chosen representative of $\mu(\alpha)$, the one we like. Henceforth we shall conform to this setting.

The isomorphism $\Der(\KK) \cong H^1(A^*,G)$ can be explained as follows (see Appendix \ref{appendix C}, Theorem \ref{C1}). For every derivation $d\in \Der(\KK)$ let $f_d:G\rightarrow A^*$ be defined as follows:
\begin{equation}\label{iso der H}
f_d(g)(a) ~:= ~ \mathrm{Tr}(g^{-1}\cdot d(g)\cdot a), \hspace{5 mm} \forall g\in G, ~ \forall a\in A,
\end{equation}
where $g$ is regarded as a matrix, say $g = (g_{i,j})_{i,j=1}^{n+1}$, and $d(g) := (d(g_{i,j}))_{i,j=1}^{n-1}$. As $d(I)$ is the null matrix, the mapping $f_d$ is a $1$-cochain of $G$ over $A^*$.

\begin{claim}\label{claim}
The mapping $f_d$ is a $1$-cocycle of $G$ over $A^*$. It is a coboundary precisely when $d$ is the null derivation.
\end{claim}
\begin{proof}
Recalling that $d(g_1g_2)  = d(g_1)g_2 + g_1d(g_2)$ for any choice of $g_1, g_2\in G$ and with the help of well known properties of the trace operator, it is straightforward to check that $f_d$ satisfies \eqref{basic formula}:
\[f_d(g_1g_2) ~ = ~ f_d(g_1)\cdot g_2 + f_d(g_2), \hspace{5 mm} \forall g_1, g_2\in G.\]
 So, $f_d\in Z^1(A^*,G)$. Turning to the second part of the claim, let $f_d\in B^1(A^*,G)$. So, there exists $\alpha\in A^*$ such that $f_d(g)(a) = \alpha(a) - \alpha(gag^{-1})$ for every $a\in A$. In particular, with $g'\in C_G(g)$ such that $\Tr(g') = \Tr(g)$ and $a = g-g'$ we have $f_d(g)(g-g') = 0.$ By~\eqref{iso der H} we obtain
\begin{equation}\label{claim eq}
\Tr(g^{-1}d(g)g) ~ = ~ \Tr(g^{-1}d(g)g').
\end{equation}
For $t\in \KK\setminus\{0\}$, let $g$ and $g'$ be the diagonal matrices of order $n+1$ with $(t, t^{-1}, 1, 1,..., 1)$ and respectively $(1, t, t^{-1}, 1,..., 1)$ as the $(n+1)$-tuples of diagonal entries. So $gg' = g'g$ and $\Tr(g) = \Tr(g')$. With $g$ and $g'$ chosen in this way we have
\[\Tr(g^{-1}d(g)g) = \Tr(d(g)) =  (1-t^{-2})d(t) ~\mbox{ and } ~ \Tr(g^{-1}d(g)g') =  (t^{-1}-1)d(t).\]
Therefore $(2-t^{-2}-t^{-1})d(t) = 0$ by \eqref{claim eq} and the above. Hence either $d(t) =0$ or $2t^2- t-1 = 0$. However $d(t) = 0$ in the latter case too, since in that case $t$ is algebraic over the minimal subfield of $\KK$. So, $d(t) = 0$ in any case, namely $d =0$.
\end{proof}

Clearly, the mapping $\delta:\Der(\KK)\rightarrow Z^1(A^*,G)/B^1(A^*,G)$ which maps every $d\in \Der(\KK)$ onto the coset $f_d+B^1(A^*,G)$ is a homomorphism of vector spaces. By Claim \ref{claim}, the homomorphism $\delta$ is injective and maps $\Der(\KK)$ onto a subspace of $H^1(A^*,G)$. Actually, $\delta$ maps $\Der(\KK)$ onto $H^1(A^*,G)$ (Appendix \ref{appendix C}, Theorem \ref{C1}). Therefore $\Der(\KK) \cong H^1(A^*,G)$ and $\delta$ provides an isomorphism from $\Der(\KK)$ to $H^1(A^*,G)$, the canonical one.

We are now ready to prove our next proposition. Our proof is actually a generalization of an argument used by Smith and V\"{o}lklein in their proof of Lemma (2.4) of \cite{SV}.

\begin{proposition} \label{lemma C}
Every non-split $1$-extension of $A$ admits $\neps$-points very well-stabilized in $G.$
\end{proposition}
\begin{proof}
Let $V:=\KK\cdot A$ be a $1$-extension of $A$, necessarily central since $G$ is perfect. With the vector space $V$ regarded as the direct product $V = \KK\times A$ of $\KK$ and $A$, the group $G$ acts as follows on $V$:
\[ (t,a)\in \KK\cdot A \mapsto (t+\phi(a,g), a\cdot g)\in \KK\cdot A, \]
where, in view of the proof of Theorem~\ref{thm link}, $\phi(\blank, g) = f_\phi(g^{-1})$ for a $1$-cocycle $f_\phi\in Z^1(A^*,G)$, with $f_\phi\in B^1(A^*,G)$ if and only if $V$ splits. So, assuming that $V$ does not split, we have $f_\phi \in Z^1(A^*,G)\setminus B^1(A^*,G)$. The value $\phi(a,g)$ taken by $\phi(\blank, g)$ at $a\in A$ is thus the value taken by $f_\phi(g^{-1})$ at $a$.

Let $[f_\phi]$ be the element of $H^1(A^*,G)$ represented by $f_\phi$. Replacing $f_\phi$ with another representative $f$ of $[f_\phi]$ amounts to replace $\phi$ with $\phi_f$, hence $V$ with a rigidly isomorphic copy of $V$. So, if $d$ is the derivation corresponding to the element $[f_\phi]$ of $H^1(A^*,G)$, up to isomorphism, we can choose $V$ so that
$f_\phi = f_d$. With this choice, $\phi = \phi_d$ where
\begin{equation}\label{phi der}
\phi_d(a, g) ~ = ~ f_d(g^{-1})(a) ~= ~ \mathrm{Tr}(g\cdot d(g^{-1})\cdot a), \hspace{5 mm} \forall g\in G, ~\forall a\in A.
\end{equation}
Given $a\in A$, let $G_{A(a)}$ be the stabilizer in $G$ of the $1$-space $A(a) := \langle a\rangle_A$ in the action of $G$ on $A$ and  $G_{V(a)}$ the stabilizer of $V(a) := \langle (0, a)\rangle_V$ in the action of $G$ on $V$. As we know, $G_{V(a)} \subseteq G_{A(a)}$. We shall prove that there exist vectors $a \in A$ which represent points $A(a)\in\neps(\Gamma)$ such that $G_{V(a)} = G_{A(a)}.$

Take as $a$ the matrix $e_{n+1,1}$ all of whose entries are null except for the entry $(n+1,1)$,
which is equal to $1.$ This matrix represents the point $\neps(p_0, H_0) \in \neps(\Gamma)$ where $p_0 \in \PG(n,\KK)$ is represented by $e_1 = (1,0,\dots,0)$ and the hyperplane $H_0$ of $\PG(n,\KK)$ is represented by the equation $x_{n+1} = 0$. The stabilizer $G_{A(e_{n+1,1})}$ of the point $A(e_{n+1,1})$ consists of the matrices $(g_{i,j})_{i,j=1}^{n+1} \in G$ with $g_{1,j} = g_{i,n+1} =0$ for $j > 1$ and $i \leq n$. If $g$ is one of these matrices then $g\cdot d(g^{-1})\cdot e_{n+1,1} = te_{n+1,1}$ for some $t\in \KK$. Consequently, with $\phi$ chosen as in \eqref{phi der}, we have
\[\phi_d(e_{n+1,1},g) ~ = ~ \mathrm{Tr}(g\cdot d(g^{-1})\cdot e_{n+1,1}) ~ = ~ 0, \hspace{5 mm} \forall g \in G_{A(e_{n+1,1})},\]
which implies that $G_{V(e_{n+1,1})} = G_{A(e_{n+1,1})}$. So far, we have proved that the point $V(e_{n+1, 1})$ is well-stabilized in $G$.

Let $\ell_V$ and $\ell'_V$ be the lines of $\PG(V)$ spanned by $(0, e_{n+1,1})$ and $(0, e_{n,1})$, respectively $(0, e_{n+1,1})$ and $(0, e_{n+1,2})$. The projection of $V$ onto $A$ maps $\ell_V$ and $\ell'_V$ onto the lines $\ell_A$ and $\ell'_A$ of $\neps(\Gamma)$ spanned by $e_{n+1,1}$ and respectively $e_{n,1}$ and $e_{n+1,2}$. Hence both $\ell_V$ and $\ell'_V$ are $\neps$-lines. An argument similar to that
exploited in the previous paragraph shows that both $\ell_V$ and $\ell'_V$ are well-stabilized.

Recall that $\Gamma$ admits two families of lines and, given a point $(p_0,H_0) \in \Gamma$, the stabilizer of $(p_0,H_0)$ in $G$ permutes transitively the lines through $(p_0,H_0)$ in each of the two families. The lines $\ell_A$ and $\ell'_A$ correspond to lines of $\Gamma$ in different families. Hence, with $G_0 := G_{A(e_{n+1,1})}$, the set $L_A := \{\ell_A\cdot g, \ell'_A\cdot g\}_{g\in G_0}$ is the full set of lines of $\neps(\Gamma)$ through $A(e_{n+1, 1})$. Put $L_V := \{\ell_V\cdot g, \ell'_V\cdot g\}_{g\in G_0}$. This is a proper subset of the set of  lines of $\PG(V)$ through $V(e_{n+1,1})$. Each of the lines in $L_V$ is well-stabilized, since it is the image by an element of $G_0$ of $\ell_V$ or $\ell'_V$, which are both well-stabilized. The projection of $V$ onto $A$ bijectively maps $L_V$ onto $L_A$. Ultimately, we have proved that $V(e_{n+1,1})$ is very well-stabilized.
\end{proof}

\begin{remark}
\em
The error mentioned in Remark \ref{rem1}, which amounts to take \eqref{iso der H} for \eqref{phi der}, spoils also the proof which Smith and V\"{o}lklein give for their Lemma 2.4, but in that context that error has no effect.
\end{remark}

\begin{corollary}\label{A2}
The extension $\widehat{A} = M\cdot A$ is $1$-universal.
\end{corollary}
\begin{proof}
This follows from Propositions \ref{lemma A} and \ref{lemma C} and Theorem \ref{1-universal}.
\end{proof}

Corollaries \ref{A1} and \ref{A2} together with Proposition \ref{lemma A} and Theorem \ref{M*} (which can be applied since $A$ is rigid, as stated in Claim \ref{rigid res}) yield the following, which completes the proof of Lemma \ref{main lemma}.

\begin{proposition}\label{A3}
We have $M^* \cong H^1(A^*,G)$.
\end{proposition}

\begin{remark}
\em
The isomorphism $M^* \cong H^1(A^*,G)$ has also been obtained by Smith and V\"{o}lklein \cite{SV} in case $n = 2$, but in a way quite different from ours. Their proof heavily relies on homological algebra. We have preferred a more direct approach.
\end{remark}

\subsection{Proof of Theorem \ref{main theorem}}

Let $\overline{A}$ be the extension of $A$ defined as in \eqref{ovA}. In \eqref{ovA} the kernel of this extension is denoted by $M$, however this can be confusing since in the previous subsection $M$ stood for the kernel of $\widehat{A}$ which, as far as we know at this stage, might be different from the kernel of $\overline{A}$
when $\dim(\Der(\KK))$ is infinite.
In order to avoid any confusion, we denote the kernel of $\overline{A}$ by $M'$.  So, $M'$ is equal to the subspace $\Der_\Omega(\KK)$ of $\Der(\KK)$ spanned by $\{d_\omega\}_{\omega\in\Omega}$ for a given derivation basis $\Omega$ of $\KK$.

The mappings from $\Omega$ to $\KK$ bijectively correspond to the derivations of $\KK$. Explicitly, a mapping $\nu:\Omega\rightarrow\KK$ corresponds to the derivation $d_\nu$ described by the following formal sum:
\[d_\nu ~ = ~ \sum_{\omega\in\Omega}\nu(\omega)d_\omega.\]
This sum does make sense since, in view of Proposition \ref{proposition B2}, for every $k\in \KK$ only a finite number of summands occur in the sum $d_\nu(k) = \sum_{\omega\in \Omega}\nu(\omega)d_\omega(k)$. On the other hand, $\nu$ also defines a linear functional $\lambda_\nu$ of $M' = \Der_{\Omega}(\KK)$, uniquely determined by the condition $\lambda_\nu(d_\omega) = \nu(\omega)$ for every $\omega\in\Omega$. Accordingly, the derivations of $\KK$ bijectively correspond to the linear functionals of $M'$. If $d$ corresponds to the linear functional $\lambda$, then for every $a\in A$ we have
\begin{equation}\label{formal sum1}
\mathrm{Tr}(g\cdot d(g^{-1})\cdot a) ~ = ~ \lambda(\sum_{\omega\in\Omega}\mathrm{Tr}(g\cdot d_\omega(g^{-1})\cdot a)d_\omega) ~=
~\sum_{\omega\in\Omega}\Tr(g\cdot d_\omega(g^{-1})\cdot a)\lambda(d_\omega).
\end{equation}
With this premise, we turn to $\overline{A}$. The extension $\overline{A}$ is obviously central.  Moreover,

\begin{lemma}\label{barA1}
The extension $\overline{A}$ is both $1$-universal and $1$-non-split.
\end{lemma}
\begin{proof}
Let $V$ be a non-split $1$-extension of $A$. Then, in view of Theorem \ref{thm link} and the isomorphism $\Der(\KK)\cong H^1(A^*,G)$, we have $V \cong V_d$ for a non-null derivation $d\in \Der(\KK)$, where $V_d =\KK\times A$ with $G$ acting on it as in \eqref{1-ext-1} with $\phi = \phi_d$ as in~\eqref{phi der}. If $\lambda$ is the linear functional corresponding to $d$ then
\begin{equation}\label{formal sum2}
\phi_d(a,g) ~ = ~ \Tr(g\cdot d(g^{-1})\cdot a) ~ = ~ \sum_{\omega\in\Omega}\mathrm{Tr}(g\cdot d_\omega(g^{-1})\cdot a)\cdot \lambda(d_\omega)
\end{equation}
for every $g\in G$ and $a\in A$. This makes it clear that $V$ is isomorphic to the quotient of $\overline{A}$ over the kernel of $\lambda$, which is a hyperplane of $M'$. Thus we have proved that $\overline{A}$ is $1$-universal.

Conversely, let $\lambda$ be a non-zero linear functional of $M'$ and let $H$ be its kernel. Then $\overline{A}/H$ is described by the mapping $\phi_d(\blank,\blank)$ defined by \eqref{formal sum2}, which is the mapping associated with the derivation $d$ corresponding to $\lambda$. Clearly $d\neq 0$, since $\lambda \neq 0$. Therefore, by Theorem \ref{thm link}, the extension $V_d$ defined by $d$ is non-split. So, $\overline{A}$ is also $1$-non-split.
\end{proof}

\begin{lemma}\label{barA2}
The extension $\overline{A}$ admits $\neps$-points very well-stabilized in $G$.
\end{lemma}
\begin{proof}
As in the proof of Proposition \ref{lemma C}, the point represented by $(0, e_{n+1,1})$ is very well-stabilized. This can be proved in just the same way as in the proof of Proposition \ref{lemma C}. For instance, in that proof we have shown that, for every derivation $d$, we have $\mathrm{Tr}(gd(g^{-1})e_{n+1,1}) = 0$ for every $g$ stabilizing $\langle e_{n+1,1}\rangle$. In particular, this holds for $d = d_\omega$ for every $\omega\in\Omega$, which is what we need to claim that this point is well-stabilized.
\end{proof}

\begin{corollary}\label{barA3}
The extension $\overline{A}$ hosts a $G$-cover of $\neps$.
\end{corollary}
\begin{proof}
In view of Lemmas \ref{barA1} and \ref{barA2}, the conclusion follows from Theorem \ref{theo 1.5}.
\end{proof}

Lemma \ref{barA1}, Corollary \ref{barA3} and Corollary \ref{1-universal ter} now yield the desired conclusion:

\begin{proposition}
 $\overline{A} = \widehat{A}$.
\end{proposition}

Formula \eqref{hatA} remains to be proved. Let now $V = V(n+1,\KK)$ and $V^*$ its dual, with the vectors of $V$ regarded as rows and those of $V^*$ as columns. So, for $\lambda\in V^*$ and $v\in V$ the tensor product $\lambda\otimes v$ is just the row-by-column product $\lambda\cdot v$ while $\lambda(v) = v\cdot \lambda = \mathrm{Tr}(\lambda\cdot v)$. The group $G$, regarded as a group of matrices, acts as follows on $V$ and $V^*$: every $g\in G$ maps $v\in V$ onto $v\cdot g$ and $\lambda\in V^*$ onto $g^{-1}\cdot\lambda$. Accordingly, $g$ maps $\lambda\cdot v$ onto $(g^{-1}\lambda)\cdot(vg) = g^{-1}(\lambda\cdot v)g$, as we know.

Let $e_1 = (1,0,\dots,0)$ and $\eta_{n+1} = (0,\dots, 0, 1)^t$. So, $\eta_{n+1}\cdot e_1 = e_{n+1,1}$ is the matrix we have considered in the proof of Proposition \ref{lemma C}. The points of $\neps(\Gamma)$ are represented by the matrices $g^{-1}e_{n+1,1}g = (g^{-1}\eta_ {n+1})\cdot(e_1g) = \lambda\cdot v$, where $\lambda = g^{-1}\eta_{n+1}$ and $v = e_1g$. The vector of $\overline{A}$ which corresponds to $\lambda\cdot v$ is
\[(\sum_{\omega\in\Omega}\mathrm{Tr}(g\cdot d_\omega(g^{-1})\cdot e_{n+1,1})d_\omega, \lambda\cdot v).\]
However
\begin{multline*} \mathrm{Tr}(g\cdot d_\omega(g^{-1})\cdot e_{n+1,1}) = \\ \mathrm{Tr}(g\cdot d_\omega(g^{-1})\cdot \eta_{n+1}e_1) = \mathrm{Tr}(d_\omega(g^{-1})\cdot\eta_{n+1}e_1g) = \mathrm{Tr}(d_\omega(g^{-1})\cdot \eta_{n+1}v)
\end{multline*}
and $d_\omega(g^{-1})\cdot\eta_{n+1} = d_\omega(g^{-1}\eta_{n+1}) = d_\omega(\lambda)$, because $g^{-1}d_\omega(\eta_{n+1}) = 0$ (since $d_\omega(\eta_{n+1})  = 0$). So,
\[\mathrm{Tr}(g\cdot d_\omega(g^{-1})\cdot e_{n+1,1}) ~ = ~ \mathrm{Tr}(d_\omega(\lambda)\cdot v) ~ = ~ v\cdot d_\omega(\lambda) ~~(= -d_\omega(v)\cdot \lambda).\]
Formula~\eqref{hatA} is proven. The proof of Theorem \ref{main theorem} is complete.
\hfill\qed

\appendix

\section{The relatively universal cover of a given embedding}\label{appendix A}

\subsection{Basics on covers of embeddings}\label{Appendix A.1}

Given a (connected) point-line geometry $\Gamma$, let $\eps:\Gamma\to\PG(V)$ and $\eps':\Gamma\to\PG(V')$ be two of its projective
embeddings. We say that $\eps'$ \emph{covers} $\eps$ (and we write $\eps' \geq \eps$) if there exists a semilinear mapping $\phi:V'\to V$ such that $\eps=\phi\circ\eps'$. The mapping $\phi$ is surjective (since $\eps'(\Gamma)$ and $\eps(\Gamma)$ span $\PG(V')$ and respectively $\PG(V)$). Moreover it is unique up to rescaling (this follows from the connectedness of $\Gamma$). We call it the  \emph{projection} of $\eps'$ onto $\eps$. Put $W=\ker{\phi}$. We say that $W$ is the {\it kernel of the projection} of $\eps'$ onto $\eps$. So $V\cong V'/W$ and we say that $\eps$ is the \emph{quotient} of $\eps'$ over $W$, also writing $\eps=\eps'/W$.

The embeddings $\eps'$ and $\eps$ are said to be \emph{isomorphic} (and we write $\eps\cong\eps'$) if $\psi$ is bijective.  If $\psi$ is not injective we say that $\eps'$ is a  \emph{proper cover} of $\eps$ and $\eps$ is a  \emph{proper quotient} of $\eps'$.

An embedding $\widehat{\eps}$ is \emph{absolutely universal} if it covers all embeddings of $\Gamma$. Clearly, the absolutely universal embedding is unique up to isomorphism. An embedding is \emph{relatively universal} if it admits no proper cover.

In general a geometry might or might not afford the absolutely universal embedding. However, Ronan in~\cite{Ronan} shows that for every
projective embedding of a point-line geometry $\varepsilon:\Gamma\to\PG(V)$ there exists an embedding
$\widehat{\varepsilon}:\Gamma\to\PG(\widehat{V})$ covering all embeddings which cover $\varepsilon$. Clearly, the embedding $\teps$ is unique up to isomorphisms and relatively universal. We call it the \emph{(relatively) universal cover} of~$\varepsilon$.

\subsection{Ronan's construction of the relatively universal cover}

Given a projective embedding $\varepsilon:\Gamma\rightarrow\PG(V)$, for every point $p$ (line $\ell$) of $\Gamma$ let $V_p$ (respectively $V_\ell$) be the $1$-dimensional ($2$-dimensional) subspace of $V$ corresponding to $\eps(p)$ (respectively $\eps(\ell)$). Denoted by $\mathcal P$ and $\mathcal L$ the set of points and the set of lines of $\Gamma$, take a copy ${\bf  V}_p$ of $V_p$ and a copy ${\bf V}_\ell$ of $V_\ell$ for every $p\in {\mathcal P}$ and every $\ell \in{\mathcal L}$. All these copies are assumed to be pairwise disjoint but for sharing the same null vector. Of course, for every $p\in {\mathcal P}$ (every $\ell \in{\mathcal L}$) an isomorphism $\iota_p:{\bf V}_p\rightarrow V_p$ (respectively $\iota_\ell:{\bf V}_\ell\rightarrow V_\ell$) is given. For every $p\in {\mathcal P}$ and every $\ell\in{\mathcal L}$ containing $p$, we put $\iota_{p,\ell} = \iota_\ell^{-1}\circ \iota_p$. Form the direct sum
\[{\bf V} ~ := ~ (\bigoplus_{p\in{\mathcal P}}{\bf V}_p)\oplus(\bigoplus_{\ell \in {\mathcal L}}{\bf V}_\ell).\]
 Let $W$ be the subspace of $\bf V$ spanned by the set
\[J ~ = ~ \{{\bf v} - \iota_{p,\ell}({\bf v}) ~:~ {\bf v}\in {\bf V}_p, ~ p\in {\mathcal P}, ~ \ell\in{\mathcal L}, ~ p\in \ell\}.\]
Put $\widehat{V} := {\bf V}/W$. Then $\widehat{V}$ hosts a projective embedding $\teps$ of $\Gamma$ and $\teps$ covers $\eps$.  Explicitly, $\teps(p) = \widehat{V}_p := ({\bf V}_p + W)/W$ and $\teps(\ell) = \widehat{V}_\ell := ({\bf V}_\ell + W)/W$ for every $p\in {\mathcal P}$ and every $\ell\in{\mathcal L}$. The natural projection $\pi_V$ of $\widehat{V}$ onto $V$, which maps ${\bf V}_p$ onto $V_p$ according to $\iota_p$ and ${\bf V}_\ell$ onto $V_\ell$ according to $\iota_\ell$, induces a projection $\widehat{\pi}:\widehat{V}\rightarrow V$ of $\teps$ onto $\eps$. Obviously, $\ker(\widehat{\pi}) = \ker(\pi_V)/W$.

As proved by Ronan \cite{Ronan}, the embedding $\teps$ constructed in this way is indeed the universal cover of~$\eps$.

\begin{proposition}\label{propA}
With $\widehat{V}$ and $\widehat{\pi}$ as above, every semilinear mapping $g$ of $V$ stabilizing $\eps(\Gamma)$ lifts to a unique semilinear mapping $\widehat{g}$ of $\widehat{V}$ such that $\widehat{\pi}(\hat{v}\cdot\widehat{g}) = \widehat{\pi}(\hat{v})\cdot g$ for every $\hat{v}\in\widehat{V}$.
\end{proposition}
\begin{proof}
Immediate, since $g$ lifts to a semilinear mapping of $\bf V$ which stabilizes both the generating set $J$ of $W$, hence $W$ itself, and the kernel $\ker(\pi_V)$ of $\pi_V$.
\end{proof}

\section{The vector space of the derivations of a field}\label{appendix B}

In this appendix we recall some information on the derivations of a field. The proofs of the results we are going to state can be found in \cite[chp. II, \S\S 12 and 17]{ZS} (also \cite[chp. VIII]{L}); what of the following is not explicitly included in \cite{ZS} or \cite{L} is anyway not too difficult to prove.

A  \emph{derivation} of a field $\KK$ is a mapping $d:\KK\rightarrow \KK$ such that
\begin{equation}\label{der-def}
d(x+y) = d(x)+d(y) ~\mbox{ and } ~ d(xy) = d(x)y+xd(y), ~~ \forall x, y \in \KK.
\end{equation}
The  \emph{null} derivation is the derivation which maps every $x\in\KK$ onto $0$. Clearly, the sum of two derivations is a derivation. Moreover, for every $k\in \KK$ and every derivation $d$ of $\KK$, the mapping $kd:x\rightarrow kd(x)$ is still a derivation. So, the derivations of $\KK$ form a $\KK$-vector space, usually denoted by $\Der(\KK)$. (The space $\Der(\KK)$ also admits a Lie algebra structure, with $[d, d'] = d\circ d'- d'\circ d$ as the bracket product, but this is not relevant for us here).

Recall that the algebraic closure in $\KK$ of a subfield $\KK'$ of $\KK$ is the subfield $\overline{\KK}'$ of $\KK$ formed by all elements of $\KK$ that are algebraic over $\KK'$. The set $\Kder(\KK)$ of the elements $x\in\KK$ such that $d(x) = 0$ for every $d\in\Der(\KK)$ contains the algebraic closure $\overline{\KK}_0$ in $\KK$ of the minimal subfield $\KK_0$ of $\KK$. When $\ch(\KK) = p > 0$ then $\Kder(\KK)$ also contains the image $\KK^p = \{x^p\}_{x\in \KK}$ of the Frobenius endomorphism of $\KK$. Note that $\overline{\KK}_0 \subseteq \KK^p$. Indeed when $\ch(\KK) = p > 0$ and $\KK$ contains
the algebraic closure of $\KK_0$, then
the subfield $\overline{\KK}_0$ is the union of all finite subfields of $\KK$ and all of them are contained in $\KK^p$.

The above already entails a complete description of $\Kder(\KK)$. Indeed, as we shall state in a few lines (Theorem \ref{der-main}), when $\ch(\KK) = 0$ then $\Kder(\KK) = \overline{\KK}_0$ while $\Kder(\KK) = \KK^p$ when $\ch(\KK) = p > 0$.

As  will be clear from the sequel, if $\dim(\Der(\KK))$ is infinite then it is uncountable. When this is the case, we can hardly exhibit even one single example of a vector-basis of $\Der(\KK)$. Neglecting the vector-bases of $\Der(\KK)$, we introduce the following notion.

A  \emph{derivation basis} of $\KK$ is a subset $\Omega$ of $\KK\setminus\Kder(\KK)$ such that for every mapping $\nu:\Omega\rightarrow\KK$ there exists a unique derivation $d_\nu\in \Der(\KK)$ such that $\nu$ is the restriction of $d_\nu$ to $\Omega$. Note that $\Omega = \emptyset$ is allowed. Of course, $\Omega = \emptyset$ when $\KK = \Kder(\KK)$, namely $\KK$ admits only the null derivation.

Recall that, given a subfield $\KK'$ of $\KK$, a subset $X$ of $\KK$ is said to  \emph{generate} $\KK$  \emph{over} $\KK'$ (to $\KK'$- \emph{generate} $\KK$, for short) if $\KK'\cup X$ generates $\KK$ as a field. Clearly, a $\KK'$-generating set $X$ of $\KK$ is minimal only if $X\cap \KK' = \emptyset$.

\begin{theorem}\label{der-main}
Every field admits derivation bases and all derivation bases of a given field have the same cardinality. Moreover,
\begin{itemize}
\item[(1)] a field $\KK$ admits the empty set as its (unique) derivation basis if and only if $\KK = \Kder(\KK)$;
\item[(2)] when $\ch(\KK) = 0$ then $\Kder(\KK)$ is the algebraic closure $\overline{\KK}_0$ in $\KK$ of the minimal subfield $\KK_0$ of $\KK$ and the derivation bases of $\KK$ are the transcendence bases of $\KK$ over $\KK_0$;
\item[(3)] when $\ch(\KK) = p > 0$ then $\Kder(\KK) = \KK^p$ and the derivation bases of $\KK$ are the minimal $\KK^p$-generating sets of $\KK$.
\end{itemize}
\end{theorem}

The common cardinality of the derivation bases of a field $\KK$ will be called the  \emph{derivation rank} of $\KK$ and denoted by $\drk(\KK)$. By claim (2) of Theorem \ref{der-main}, when $\ch(\KK) = 0$ then $\drk(\KK)$ is the transcendence degree of $\KK$ over $\KK_0$. When $\ch(\KK) = p > 0$ then $\drk(\KK)$ is the $\KK^p$-generating rank of $\KK$, namely the common size of the minimal $\KK^p$-generating sets. Given a derivation basis $\Omega$ of $\KK$, when $\ch(\KK) = 0$ every $k\in \KK$ belongs to the algebraic closure in $\KK$ of the subfield of $\KK$ generated by $\Kder(\KK)\cup\Omega_k$ for a finite (possibly empty) subset $\Omega_k$ of $\Omega$. When $\ch(\KK) = p >0$ every $k\in \KK$ belongs to the subfield of $\KK$ generated by $\Kder(\KK)\cup\Omega_k$ for a finite subset $\Omega_k$ of $\Omega$. In either case, with the derivation $d_\omega$ (for $\omega \in \Omega$) defined as in Section \ref{Intro2}, we have $d_\omega(k)\neq 0$ only if $\omega \in \Omega_k$. Consequently,

\begin{proposition}\label{proposition B2}
For every $k\in \KK$, we have $d_\omega(k) = 0$ for all but finitely many (possibly no) choices of $\omega\in\Omega$.
\end{proposition}

Obviously, if $\Omega$ is a derivation basis of $\KK$ then $\Der(\KK)$ is isomorphic with the $\KK$-vector space $\KK^\Omega$ of all mappings from $\Omega$ to $\KK$. So, if $\drk(\KK)$ is finite then $\drk(\KK) = \dim(\Der(\KK))$, otherwise $\drk(\KK) < \dim(\Der(\KK))$.

\section{The isomorphism $H^1(\mathfrak{sl}(n+1,\KK)^*,\mathrm{SL}(n+1,\KK))\cong \Der(\KK)$}\label{appendix C}

Let $A$ be the adjoint module for the group $G = \mathrm{SL}(n+1,\KK)$ and let $A^*$ be its dual. Consider the following mapping from $\Der(\KK)$ to the space $C^1(A^*, G)$ of the $1$-cochains of $G$ over $A^*$:
\begin{equation}\label{eq C1}
\delta_{A^*} ~ \left\{\begin{array}{rcc}
 \Der(\KK) & \longrightarrow & C^1(A^*,G), \\
& & \\
d\in \Der(\KK) & \longrightarrow &
\delta_{A^*}(d) ~:~ \left(\begin{array}{rcl}
& g \in G & \\
& \downarrow &  \\
a\in A & \rightarrow & \Tr(g^{-1}d(g)a)
\end{array}\right).
\end{array}\right.
\end{equation}
As proved earlier in this paper (Claim \ref{claim}) for every $d\in \Der(\KK)$ the mapping $\delta_{A^*}(d):G\rightarrow A^*$ is indeed a $1$-cocycle of $G$ over $A^*$ and $\delta_{A^*}(d)$ is a coboundary if and only if $d = 0$. Clearly, the mapping $\delta_{A^*}:\Der(\KK)\rightarrow Z^1(A^*,G)$ is linear and injective. Therefore the mapping
\[\overline{\delta}_{A^*} ~ \left\{\begin{array}{ccc}
\Der(\KK) & \longrightarrow & H^1(A^*,G) \\
d\in \Der(\KK) & \rightarrow & \delta_{A^*}(d) + B^1(A^*,G)
\end{array}\right.\]
is an injective linear mapping from $\Der(\KK)$ to $H^1(A^*,G) = Z^1(A^*,G)/B^1(A^*,G)$. It is an isomorphism if and only if it is surjective. Equivalently, $\overline{\delta}_{A^*}$ is an isomorphism if and only if $\delta_{A^*}$ maps $\Der(\KK)$ onto a complement of $B^1(A^*,G)$ in $Z^1(A^*,G)$.

In this appendix, with the help of a celebrated result of Taussky and Zassenhaus (Theorem 1 of \cite{TZ}) we shall prove the following.

\begin{theorem}\label{C1}
Let $n \geq 2$. Then the mapping $\delta_{A^*}$ defined as in {\rm \eqref{eq C1}} is an isomorphism from $\Der(\KK)$ to a complement of $B^1(A^*,G)$ in $Z^1(A^*,G)$.
\end{theorem}

\begin{remark}
\em
The case $n = 2$ of Theorem \ref{C1} is covered by Lemma (2.2) of Smith and V\"{o}lklein \cite{SV}. However the proof that Smith and
V\"{o}lklein offer for that lemma is not very clear. They show that if $n = 2$ then $\Der(\KK) \cong H^1(A^*,G)$, obtaining this conclusion from the main theorem of V\"{o}lklein \cite{V} and a few more results from the literature to fix the case where $|\KK| \leq 9$, but they do not explain why this isomorphism is provided by the mapping $\overline{\delta}_{A^*}$. If $\dim(\Der(\KK)) < \infty$ then $\overline{\delta}_{A^*}$, being linear and injective, is also surjective and, therefore, it is an isomorphism. However $\dim(\Der(\KK))$ can be infinite. When this is the case, knowing that $\Der(\KK) \cong H^1(A^*,G)$ is of no help.
\end{remark}

\subsection{Preliminaries}

With $G$, $A$ and $A^*$ as above, let $M := M_{n+1}(\KK)$ be the vector space of square matrices of order $n+1$ with entries in $\KK$, regarded as a $G$-module with $G$ acting by conjugation on it:
\[g\in G~ : ~m\in M ~\rightarrow ~m\cdot g := g^{-1}mg.\]
Clearly, the space $C^1(M,G)$ of $1$-cochains of $G$ over $M$ properly contains $C^1(A,G)$. Nevertheless,

\begin{lemma}\label{C00}
We have $Z^1(M,G) = Z^1(A,G)$.
\end{lemma}
\begin{proof}
Let $f\in Z^1(M,G)$. The equality $f(g_1g_2) = f(g_1)\cdot g_2 + f(g_2)$ implies $\Tr(f(g_1g_2)) = \Tr(f(g_1)) + \Tr(f(g_2))$. Since $G$ is perfect, the latter equality implies that $\Tr(f(g)) = 0$ for every $g\in G$. Therefore $f(g)\in A$, since $A$ is the subspace of $M$ formed by the matrices $m\in M$ with $\Tr(m) = 0$.
\end{proof}

Let $\delta_M:\Der(\KK)\rightarrow C^1(M,G)$ be defined as follows:
\begin{equation}\label{eq C4}
\delta_M ~  \left\{\begin{array}{rcl}
 \Der(\KK) & \longrightarrow & C^1(M,G),\\
& & \\
d\in \Der(\KK) & \longrightarrow &
\delta_M(d) : g\in G \rightarrow g^{-1}d(g) \in M.
\end{array}\right.
\end{equation}
It is straightforward to check that $\delta_M(d)$ is indeed a $1$-cocycle of $G$ over $M$ for every $d\in \Der(\KK)$. Clearly, $\delta_M$ is an injective linear map from $\Der(\KK)$ to $Z^1(M,G)$. The following is Theorem 1 of Taussky and Zassenhaus \cite{TZ}.

\begin{theorem}\label{C5}
When either $n \geq 2$ or $\KK \neq \FF_3$ the mapping $\delta_M$ defined as in {\rm \eqref{eq C4}} is an isomorphism from $\Der(\KK)$ to a complement of $B^1(M,G)$ in $Z^1(M,G)$. When $n = 1$ and $\KK = \FF_3$ then $\Der(\KK) = \{0\}$ but $|H^1(M,G)|  = 3$.
\end{theorem}

The $1$-dimensional subspace $K := \langle I\rangle$ of $M$ generated by the identity matrix $I\in M$ is the centralizer $C_M(G)$ of $G$ in $M$. Given $m\in M$, we denote the coset $m+K$ by the symbol $[m]$. As explained in Section \ref{sec 3.1} (formula \eqref{Astar-M/I}), the mapping
\[\kappa ~ \left\{\begin{array}{rcl}
M/K & \longrightarrow & A^*\\
{[m]}\in M/K & \longrightarrow & \kappa({[m]}) : \left(\begin{array}{c}
a\in A \\
\downarrow\\
\Tr(ma)
\end{array}\right)
\end{array}\right.\]
is well defined and provides an isomorphism of $G$-modules from $M/K$ to the dual $A^*$ of $A$. Accordingly, $\kappa$ induces an isomorphism from $Z^1(M/K,G)$ to $Z^1(A^*,G)$ which maps $B^1(M/K,G)$ onto $B^1(A^*,G)$. The composition $\delta_{M/K} := \kappa^{-1}\circ\delta_{A^*}$ maps every $d\in\Der(\KK)$ onto the cocycle $\delta_{M/K}(d)\in Z^1(M/K,G)$ acting as follows:
\[\delta_{M/K}(d)~:~ g\in G ~ \rightarrow ~ [g^{-1}d(g)] \in M/K.\]
So, Theorem \ref{C1} is equivalent to the following.

\begin{theorem}\label{C6}
Let $n\geq 2$. Then the mapping $\delta_{M/K}$ defined as above is an isomorphism from $\Der(\KK)$ to a complement of $B^1(M/K, G)$ in $Z^1(M/K, G)$.
\end{theorem}

The rest of this subsection is devoted to a proof of Theorem \ref{C6}. We need to distinguish two cases: the {\em plain case}, where either $\ch(\KK) = 0$ or $\ch(\KK) = p > 0$ but $p$ does not divide $n+1$ and the {\em singular case}, where $\ch(\KK) = p > 0$ divides $n+1$.

\subsection{The plain case}

In this case $\Tr(I) \neq 0$. Therefore every coset of $K$ in $M$ meets $A$ in a unique element. Accordingly, $M/K$ and $A$ are isomophic as $G$-modules. Morover, for $m_1, m_2\in M$, if $[m_1] = [m_2]$ then $m_1-m_1\cdot g = m_2-m_2\cdot g$ for every $g\in G$. Therefore every coboundary $f_m: g\in G \rightarrow m-m\cdot g$ of $G$ over $M$ can be represented by a matrix $m\in A$. It follows that $B^1(M,G) = B^1(A,G)$. Since $Z^1(M,G) = Z^1(A,G)$ as stated in Lemma \ref{C00}, we immediately obtain Theorem \ref{C6} from Theorem \ref{C5} and the isomorphism $M/K \cong A$.

\subsection{The singular case}

Now $\ch(\KK) = p > 0$ divides $n+1$. Consequently $\Tr(I) = 0$ and, therefore, for every $[m]\in M/K$ all matrices of $[m]$ have the same trace. We put $\Tr([m]) := \Tr(m)$ for $m\in [m]$.  Given $F\in Z^1(M/K, G)$, a $1$-cochain $f\in C^1(M,G)$ {\em represents} $F$ is $[f] = F$, namely $[f(g)] = F(g)$ for every $g\in G$.

\begin{lemma}\label{C01}
We have $Z^1(M/K, G) =Z^1(A/K, G)$.
\end{lemma}
\begin{proof}
Let $F\in Z^1(M/K,G)$. Then $\Tr(F(g_1g_1)) = \Tr(F(g_1)) + \Tr(F(g_2))$ by the cocycle condition on $F$ for any choice of $g_1, g_2\in G$. So, the map $g\in G \rightarrow \Tr(F(g))$ is a homomorphism from $G$ to the additive group of $\KK$. However $G$ is perfect. Therefore $\Tr(F(g)) = 0$ for every $g\in G$. In other words, if $f\in C^1(M/K,G)$ represents $F$, then $\Tr(f(g)) = 0$ for every $g\in G$, namely $f\in C^1(A,G)$.
\end{proof}

\begin{lemma}\label{C02}
Every $F\in Z^1(M/K,G)$ admits at most one representative in $Z^1(M,G)$.
\end{lemma}
\begin{proof}
Let $f_1, f_2\in Z^1(M,G)$ be such that $[f_1] = [f_2]$. Then there exists a mapping $\lambda : G\rightarrow \KK$ such that $f_1(g)-f_2(g) = \lambda(g)I$ for every $g\in G$. However $f_1-f_2\in Z^1(M,G)$. Hence the mapping $f_\lambda : g\in G \rightarrow \lambda(g)I$ is a cocycle. The cocycle condition on $f_\lambda$ forces $\lambda(g_1g_2) = \lambda(g_1)+\lambda(g_2)$. Therefore $\lambda$ is the null mapping, since $G$ is perfect.
\end{proof}

The following is obvious.

\begin{corollary}\label{C03}
If every $F\in Z^1(M/K,G)$ admits a  representative in $Z^1(M,G)$ then the function which maps every $F\in Z^1(M/K, G)$ onto its representative in $Z^1(M,G)$ (unique by Lemma \ref{C02}) is an isomorphism from $Z^1(M/K,G)$ to $Z^1(M,G)$ and maps $B^1(M/K, G)$ onto $B^1(M,G)$.
\end{corollary}

\begin{remark}
\em
Under the hypotheses of Corollary \ref{C03} we have
\[Z^1(A/K,G) ~ = ~  Z^1(M/K, G) ~ \cong ~  Z^1(M,G) ~ = ~ Z^1(A,G)\]
and the isomorphism from $Z^1(M/K, G)$ to $Z^1(M,G)$ induces an isomorphism from $B^1(M/K, G) = B^1(A/K,G)$ to $B^1(M,G)$. Hence
\[H^1(A/K, G) ~ = ~  H^1(M/K,G) ~ \cong ~ H^1(M,G).\]
However, under the hypotheses we have made on $\KK$, the space $H^1(M,G)$ is a proper homorphic image of $H^1(A, G)$, with kernel of dimension $1$. Indeed $B^1(A,G) \subset B^1(M,G)$ has codimension $1$ in $B^1(M,G)$, the complements of $B^1(A,G)$ in $B^1(M,G)$ being provided by the $\KK$-spans of the coboundaries $f_{m}:g\in G\rightarrow m-m\cdot g$ for $m\in M$ such that $\Tr(m) \neq 0$.
\end{remark}

If the hypothesis of Corollary \ref{C03} is satisfied then the statement of Theorem \ref{C6} immediately follows from Theorem \ref{C5}. So, in order to obtain Theorem \ref{C6} we only need to prove the following:

\begin{claim}\label{C7}
Every $F\in Z^1(M/K,G)$ admits a representative in $Z^1(M,G)$.
\end{claim}
\begin{proof}
Note firstly that if $n = 2$ and $\KK = \FF_3$ then $H^1(M/K, G) = \{0\}$, as Smith and V\"{o}lklein show in the proof of Lemma (2.2) of \cite{SV}. On the other hand, $\Der(\FF_3) = \{0\}$. So, in this case there is nothing to prove. Accordingly, henceforth we assume that $(n, \KK) \neq (2, \FF_3)$.

As in proof of Claim \ref{rigid res}, we denote by $e_{i,j}$ the square matrix with all null entries except the $(i.j)$-entry, which is $1$. Also, $u_{i,j}(t) := I+te_{i,j}$. So, $U = \{u_{i,j}(t)~|~ i\neq j, ~ t\in \KK\}$ is a (redundant) set of generators of $G$.

Put $J := \{1, 2,..., n+1\}$ and, given $k\in J$, let $J_k := J\setminus\{k\}$. We set $U_k := \{u_{i,j}(t)~ | ~ t\in \KK,~ i, j \in J_k, ~ i\neq j\}$ and denote by $G_k$ and $M_k$ the subgroup of $G$ generated by $U_k$ and respectively the subspace of $M$ generated by $\{e_{i,j}\}_{i,j\in J_k}$. Clearly $M_k\cong M_n(\KK)$ and $G_k \cong \mathrm{SL}(n,\KK)$ is the derived subgroup of the stabilizer in $G$ of both the subspace $M_k$ of $M$ and the matrix $e_{k,k}\in M$.

Let $\pi_k:M\rightarrow M_k$ be the projection of $M$ onto $M_k$ which maps every matrix $m = \sum_{i,j\in J}m_{i,j}e_{i,j}$ of $M$ onto $|m|_k := \pi_k(m) = \sum_{i,j\in J_k}m_{i,j}e_{i,j}$. With this notation $g = |g|_k+e_{k.k}$ for every $g\in G_k$ and $|m\cdot g|_k = |m_k|\cdot g =  |m_k|\cdot |g|_k$ for every $m\in M$ and every $g\in G_k$. If we want to be pedantic, the groups $G_k$ and $\pi_k(G_k)$ are different but, since they are isomorphic in an obvious way, we shall regard them as the same object. Accordingly, if $g\in G_k$ we shall omit to distinguish between $g$ and $|g|_k$. The obvious isomorphisms $G_k\cong \mathrm{SL}(n,\KK)$ and $M_k\cong M_n(\KK)$ provide an isomorphism of modules from $(M_k,G_k)$ to $(M_n(\KK), \mathrm{SL}(n,\KK))$. Also, if $K_k = \pi_k(K)$ is the $\KK$-span of $|I|_k$ in $M_k$ and $K'$ is the subspace of $M_n(\KK)$ spanned by the identity matrix of $M_n(\KK)$, then the isomorphism from $(M_k,G_k)$ to $(M_n(\KK), \mathrm{SL}(n,\KK))$ induces an isomorphism from $(M_k/K_k,G_k)$ to $(M_n(\KK)/K', \mathrm{SL}(n,\KK))$.

Let $F\in Z^1(M/K, G)$. Given $k\in J$ the restriction $F_k: g\in G_k\rightarrow \pi_k(F(g))\in M_k/K_k$ of $F$ to $G_k$ is a cocycle of $G_k$ over $M_k/K_k$. The characteristic $p$ of $\KK$ does not divide $n$, as $p$ divides $n+1$ by assumption. Therefore the module $(M_k/K_k, G_k) \cong (M_n(\KK), \mathrm{SL}(n,\KK))$ falls into the plain case, but with $n$ replaced by $n-1$. We have already proved that the conclusion of Theorem \ref{C6} holds true in the plain case. (That conclusion holds true also if $n = 2$, since our assumption that $(n,\KK) \neq (2,\FF_3)$ forbids the exceptional case of Theorem \ref{C1} to occur for $(M_k, G_k)$). Therefore there exist a unique derivation $d_k\in \Der(\KK)$ and a representative $f_k\in C^1(M,G)$ of $F$ such that $|f_k(g)|_k = g^{-1}d_k(g)$ for every $g\in G_k$. Accordingly, $f_k(g) = g^{-1}d_k(g) + m_k(g)$ for a suitable matrix $m_k(g) = m_{k,k}(g)e_{k,k} + \sum_{i\neq k}(m_{i,k}(g)e_{i,k} + m_{k,i}(g)e_{k,i})$.

We have $\Tr(f_k(g)) = \Tr(g^{-1}d_k(g)) + m_{k,k}(g)$. However $\Tr(g^{-1}d_k(g)) = 0$ because $g^{-1}d_k(g) \in M_k$ ($\subset M$), the mapping $\delta_{M_k}(d_k):g\in G_k\rightarrow g^{-1}d_k(g)$ is a cocycle of $G_k$ over $M_k$ and $Z^1(M_k,G_k) = Z^1(A_k, G_k)$ by Lemma \ref{C00}, where $A_k \subseteq M_k$ is the subspace formed by the traceless matrices of $M_k$. Moreover $\Tr(f_k(g)) = 0$ by Lemma \ref{C01} (on $(M, G)$). Therefore $m_{k,k}(g) = 0$. So,
\begin{equation}\label{eq C5}
f_k(g) ~ = ~ g^{-1}d_k(g) + \sum_{i\neq k}(m_{i,k}(g)e_{i,k} + m_{k,i}(g)e_{k,i}).
\end{equation}
We call $f_k$ a $k$-{\em partial cocycle}. If $f_k$ and $f'_k$ are two $k$-partial cocycles which both represent $F$, then $f_k(g)-f'_k(g) \in K$ for every $g\in G$. As $d_k$ is uniquely determined by $k$, equation \eqref{eq C5} implies that $f_k(g) = f'_k(g)$ for every $g\in G_k$.

So far we have assumed only that $n \geq 2$ but from here on we need to distinguish the case of $n \geq 3$ from the case $n = 2$. \\

\noindent
{\bf Case 1}.  Let $n \geq 3$. Let $g = u_{i,j}(t) \in U$ and, for $k, h \in J\setminus\{i,j\}$ and let $f_k$ and $f_h$ be a $k$-partial cocycle and a $h$-partial cocycle which both represent $F$. As $u_{i,j}(t) \in G_k\cap G_h$, have $f_k(u_{i,j}(t)) - f_h(u_{i,j}(t)) \in K$. Recall that $u_{i,j}(t)^{-1} = I-te_{i,j}$ and $d(u_{i,j}(t)) = d(t)e_{i,j}$ for every derivation $d$. Hence $u_{i,j}(t)^{-1}d(u_{i,j}(t)) = d(u_{i,j}(t)) =  d(t)e_{i,j}$ for every derivation $d$. In view of this fact and equation \eqref{eq C5}, the condition $f_k(u_{i,j}(t)) - f_h(u_{i,j}(t)) \in K$ is equivalent to the following, where $\lambda$ is a suitable mapping from $\KK$ to $\KK$ and we write $m_{r,k}$, $m_{k,r}$, $m_{s,h}$, $m_{h,s}$ instead of $m_{r,k}(u_{i,j}(t))$, $m_{k,r}(u_{i,j}(t))$, $m_{s,k}(u_{i,j}(t))$ and $ m_{k,s}(u_{i,j}(t))$, for short.
\begin{equation}\label{eq C6}
(d_k(t)-d_h(t))e_{i,j} + \sum_{r\neq k}(m_{r,k}e_{r,k} + m_{k,r}e_{k,r}) -  \sum_{s\neq h}(m_{s,h}e_{s,h} + m_{h,s}e_{h,s}) = \lambda(t) I.
\end{equation}
This equation implies $\lambda(t) = 0$ and $d_k(t) = d_h(t)$ for every $t\in \KK$. Hence $d_k = d_h$, namely $d_k$ does not depend on the particular choice of $k$. Accordingly, henceforth we write $d$ instead of $d_k$.

Equation \eqref{eq C6} also implies that $m_{r,k}(g) = m_{k,r}(g) = 0$ for $r\not\in \{k,h\}$ and $m_{s,h}(g) = m_{h,s}(g) = 0$ for $s\not\in
\{h,k\}$. So, for $g = u_{i,j}(t)$ and $\{k,h\}\cap\{i,j\} = \emptyset$ equation \eqref{eq C5} boils down to the following:
\begin{equation}\label{eq C7}
f_k(g) ~ = ~ g^{-1}d(g) + m_{h,k}e_{h,k} + m_{k,h}(g)e_{k,h}.
\end{equation}
Given $k\not\in \{i,j\}$, equation \eqref{eq C7} holds for every $h\not\in\{i,j,k\}$. If $n > 3$ then, given $k\not\in\{i,j\}$, at least two choices are left for $h\not\in \{i,j,k\}$. Consequently \eqref{eq C7} implies that $m_{h,k}(g) = m_{k,h}(g) = 0$, namely $f_k(g) = g^{-1}d(g)$. In this case $F$ admits a representative $f\in C^1(M,G)$ such that $f(g) = \delta_M(d)(g) = g^{-1}d(g)$ for every $g\in U$. As $U$ is a set of generators for $G$, the cocycle $\delta_M(d)$ is a representative of $F$. So, Claim \ref{C7} is proved. Actually, we have proved more than that claim. Indeed we have proved just the statement of Theorem \ref{C6}.

Let now $n = 3$. Then \eqref{eq C7} shows that, for every partition $\{\{i,j\},\{k,h\}\}$ of $\{1,2,3,4\}$, the cocycle $F$ admits a representative $f_{\{k,h\}}\in C^1(M,G)$ such that
\begin{equation}\label{eq C8}
f_{\{k,h\}}(u_{i,j}(t)) ~ = ~ d(t)e_{i,j} + m_{h,k;i,j}(t)e_{h,k} + m_{k,h:i,j}(t)e_{k,h}
\end{equation}
for suitable functions $m_{h,k;i,j}, m_{k,h;i,j} : \KK \rightarrow \KK$. Since $\{i,j\}\cap\{k,h\} = \emptyset$, we have $u_{i,j}(t)u_{k,h}(s) = u_{k,s}(s)u_{i,j}(t)$. By \eqref{eq C8} and the cocycle condition on $F$, this equality implies that there exist a mapping $\lambda:\KK\times \KK \rightarrow \KK$ such that
\begin{equation}\label{eq C9}
\begin{array}{rcl}
sm_{h,k;i,j}(t)(e_{h,h} - e_{k,k}) -s^2m_{h,k;i,j}(t)e_{k,h} + & & \\
+ tm_{j,i;k,h}(s)(e_{i,i}- e_{j,j}) + t^2m_{j,i;k,h}(s)e_{i,j} & = & \lambda(t,s)I.
\end{array}
\end{equation}
Therefore $s^2m_{h,k;i,j}(t) = t^2m_{j,i;k,h}(s) = 0$ for every choice of $s,t\in \KK$, namely $m_{h,k;i,j}(t) = m_{j,i;k,h}(s) = 0$ for any $s, t\in \KK$. Consequently, $\lambda(t,s) = sm_{h,k;i,j}(t) = tm_{j,i;k,h}(s) = 0$. By the same argument but with $u_{k,h}(s)$ replaced by $u_{h,k}(s)$ we also obtain that $m_{k,h;i,j}(t) = 0$ for every $t\in \KK$. Accordingly, \eqref{eq C8} just says that $f_{\{k,h\}}(u_{i,j}(t)) = d(t)e_{i,j}$. The conclusion now follows in the same way as when $n > 3$. \\

\noindent
{\bf Case 2.} Let $n = 2$.  Formula \eqref{eq C5} is still valid but we cannot use it to compare $f_k(u_{i,j}(t))$ with $f_h(u_{i,j}(t))$ for $h\not \in \{i,j,k\}$, since now $\{i,j,k\} = \{1,2,3\} = J$. We shall exploit commutation relations on pairs of elements of the set $U$ of generators we have chosen for $G$. In view of this task, having introduced a bit more notation will be helpfull.

Given two matrices $m_1, m_2\in M$, if $m_1-m_2\in K$ then we write $m_1\equiv m_2$. With this notation, given three representatives $f, f', f''$ of $F$ in $C^1(M,G)$ ad two elements $g_1, g_2$ of $G$, the cocycle condition on $F$ implies that
\begin{equation}\label{eq C10}
f(g_1)\cdot g_2 + f'(g_2) ~ \equiv ~ f''(g_1g_2).
\end{equation}
Consequently, if $g_1g_2 = g_2g_1$ then
\begin{equation}\label{eq C11}
f(g_1)\cdot g_2 + f'(g_2) ~ \equiv ~ f'(g_2)\cdot g_1 + f(g_1).
\end{equation}
From \eqref{eq C10} with $g_1 = g$ and $g_2 = g^{-1}$ we also obtain
\begin{equation}\label{eq C12}
f'(g^{-1}) ~ \equiv ~ -f(g)\cdot g^{-1}.
\end{equation}
Let $(g_1, g_2) = g_1g_2g_1^{-1}g_2^{-1}$ be the commutator of $g_1$ and $g_2$. By \eqref{eq C10} and \eqref{eq C12} we obtain
\begin{equation}\label{eq C13}
f''((g_1, g_2)) \equiv  f(g_1)\cdot g_2g_1^{-1}g_2^{-1} + (f'(g_2) - f(g_1))\cdot g_1^{-1}g_2^{-1} - f'(g_2)\cdot g_2^{-1}.
\end{equation}
With the help of \eqref{eq C11}, recalling that $u_{i,j}(t)$ and $u_{k,h}(s)$ commute when either $j = h$ or $i = k$, from \eqref{eq C5} we obtain that for every ordering $(i,j,k)$ of $\{1,2,3\}$ there exist scalars $a_{i,j}, b_{i,j}\in \KK$ such that
\begin{equation}\label{eq C14m}
f_k(u_{i,j}(t)) ~ = ~ d_k(t)e_{i,j} + a_{i,j}(te_{k,i} - t^2e_{k,j}) + b_{i,j}(t^2e_{i,k} + te_{j,k}), \hspace{3 mm}\forall t\in \KK.
\end{equation}
Moreover $a_{i,j} + a_{i,k} = 0$ and $b_{i,j} + b_{k,j} = 0$ for every ordering $(i,j,k)$ of $\{1,2,3\}$. Recall now that also $(u_{i,j}(t), u_{j,k}(s)) = u_{i,k}(ts)$ for every ordering $(i,j,k)$ of $\{1,2,3\}$. With elementary but laborious calculations from \eqref{eq C13} we obtain that $a_{i,j} = 0$,  $b_{i,j} = b_{i,k} = 0$ and
\begin{equation}\label{eq C14}
d_k(t) + d_i(s) ~= ~ d_j(ts), \hspace{5 mm}\forall t,s\in \KK.
\end{equation}
Equation \eqref{eq C14} implies that $d_k = d_i = d_j =  d$ for a unique derivation $d\in \Der(\KK)$, while the conditions $a_{i,j} = b_{i,j} = 0$ imply that
\[f_k(u_{i,j}(t)) ~ = ~ d(t)e_{i,j} ~ (= ~ (I-te_{i,j})d(t)e_{i,j} ~ = ~ u_{i,j}(t)^{-1}d(t)e_{i,j}).\]
As in the previous cases, $[\delta_M(d)(u)] = F(u)$ for every generator $u\in U$ of $G$. The conclusion follows.
\end{proof}

The proof of Theorem \ref{C6} is complete. Thus, Theorem \ref{C1} is proven.

 \end{document}